\documentclass[11pt,reqno]{amsart}

\usepackage{amsmath,amsthm,amssymb,mathtools,tocvsec2,pdflscape,cite}
\usepackage{latexsym}
\usepackage{color,soul}
\usepackage[utf8]{inputenc}
\usepackage[T1]{fontenc}
\usepackage{tikz}
\usepackage{cancel}
\usepackage[tmargin=1.20in,bmargin=1.19in,rmargin=1.2in,lmargin=1.2in]{geometry}
\usepackage[breaklinks=true]{hyperref}

\allowdisplaybreaks

\theoremstyle{definition}

\newtheorem{defn}[equation]{Definition}

\newtheorem{remark}[equation]{Remark}
\newtheorem{example}[equation]{Example}

\newtheorem{ualgorithm}{\textrm{\textbf{Algorithm}}}
     
\newtheorem{unotion}[ualgorithm]{\textrm{\textbf{Notion}}}

\newtheorem{uclaim}{\textrm{\textbf{Claim}}}

\newtheorem{uquestion}{\textrm{\textbf{Question}}}

\theoremstyle{plain}
\newtheorem{lemma}[equation]{Lemma}
\newtheorem{theorem}[equation]{Theorem}

\newtheorem{utheorem}{\textrm{\textbf{Theorem}}}

\newtheorem{ucor}[utheorem]{\textrm{\textbf{Corollary}}}

\newtheorem{ulemma}[utheorem]{\textrm{\textbf{Lemma}}}

\newcommand{\Z}{\mathbb{Z}}

\newcommand{\E}{\mathrm{E}}

\newcommand{\A}{\mathbb{\alpha}}
\newcommand{\B}{\mathbb{\beta}}

\newcommand{\diag}{\mbox{diag}}

\newcommand{\I}{\mathrm{I}}
\newcommand{\El}{\mathrm{E}}
\newcommand{\adj }{\mathrm{adj}}
\newcommand{\Ro }{\mathrm{R}}
\newcommand{\Co }{\mathrm{C}}
\newcommand{\So }{\mathrm{S}}
\newcommand{\ml }{\mathrm{m}}

\title{}
\numberwithin{equation}{section}
\begin{document}

\title[Inequalities for totally nonnegative matrices]{Inequalities for totally nonnegative matrices: Gantmacher--Krein, Karlin, and Laplace}

\author[Shaun M. Fallat \and Prateek Kumar Vishwakarma]{Shaun M. Fallat \and Prateek Kumar Vishwakarma}

\address[]{Shaun M. Fallat}
\email{shaun.fallat@uregina.ca}

\address[]{Prateek Kumar Vishwakarma}
\email{prateek.vishwakarma@uregina.ca,~prateekv@alum.iisc.ac.in}

\address{Department of Mathematics and Statistics$,$ College West 307$.$14, 3737 Wascana Parkway$,$ University of Regina$,$ SK S$4$S $0$A$2,$ Canada}

\address{\today}



\begin{abstract}
A real linear combination of products of minors which is nonnegative over all totally nonnegative (TN) matrices is called a determinantal inequality for these matrices. It is referred to as multiplicative when it compares two collections of products of minors and additive otherwise. Set theoretic operations preserving the class of TN matrices naturally translate into operations preserving determinantal inequalities in this class. We introduce index-row (and index-column) operations that act directly on all determinantal inequalities for TN matrices, and yield further inequalities for these matrices. These operations assist in revealing novel additive inequalities for TN matrices embedded in the classical identities due to Laplace $[$\textit{Mem$.$ Acad$.$ Sciences Paris} $1772]$ and Karlin $(1968).$ In particular, for any square TN matrix $A,$ these derived inequalities generalize -- to every $i^{\mbox{th}}$ row of $A$ and $j^{\mbox{th}}$ column of ${\rm adj} A$ -- the classical Gantmacher--Krein fluctuating inequalities $(1941)$ for $i=j=1.$ Furthermore, our index-row/column operations reveal additional undiscovered fluctuating inequalities for TN matrices. 

The introduced index-row/column operations naturally birth an algorithm that can detect certain determinantal expressions that do not form an inequality for TN matrices. However, the algorithm completely characterizes the multiplicative inequalities comparing products of pairs of minors. Moreover, the underlying index-row/column operations add that these inequalities are offshoots of certain ``complementary/higher'' ones. These novel results seem very natural, and in addition thoroughly describe and enrich the classification of these multiplicative inequalities due to Fallat--Gekhtman--Johnson $[$\textit{Adv$.$ Appl$.$ Math$.$} $2003]$ and later Skandera $[$\textit{J$.$ Algebraic Comb$.$} $2004]$ {(and revisited by Rhoades--Skandera $[$\textit{Ann$.$ Comb$.$} $2005,$ \textit{J$.$ Algebra} $2006]$)}.
\end{abstract}

\maketitle
\tableofcontents

\section{Introduction}
A real matrix with all its minors nonnegative is called totally nonnegative (TN). These matrices arise in applied and pure mathematics, including statistics, numerical analysis, mathematical physics, computational interpolation theory, and combinatorics \cite{A, BFZ, B, FJb, FG05, FGJ00, FZ2, GK1, GK60, GK2, GM, K2, karmac, Lu1, Sc,W}. One of the most interesting and fruitful developments within the theory of total nonnegativity has been the connection between totally nonnegative matrices and their associated matrix factorizations, especially the well-known bidiagonal factorization \cite{C1, C2, Fal01, FZ2, GP1, GP4}. There is a natural correspondence between these bidiagonal factorization and the so-called planar networks \cite{Fal01, FZ2}. A significant utility of these planar networks in this context can be derived from Lindstrom's lemma \cite{Lin73} which connects the minors of a given matrix to weighted sums of families of directed paths. These weighted paths are very desirable combinatorial objects owing to their natural yet crucial role in the development of determinantal inequalities and identities for totally nonnegative matrices \cite{BF08, FGJ03, GS, RS2005, RS2006, Skan-Sos}. See Section~\ref{Section-TN} for a brief discussion on these connections that are relevant to this paper.

Studies on determinantal inequalities have intertwined with the continued development of the theory of totally nonnegative matrices. In particular, multiplicative determinantal inequalities, which are inequalities comparing products of minors, have been studied extensively with respect to this class of matrices \cite{BF08,FGJ03, RS2005, RS2006, GS, Ska04}, and are often found to align with known inequalities for positive semidefinite matrices and even $M$-matrices \cite{BJ2, FHJ, FJ}. For example, the topic originates with the verification of the classical inequalities due to Hadamard, Fischer, and Koteljanskii \cite{Ko1, Ko2}. Much of this general study of determinantal inequalities for various positivity classes of matrices, including totally nonnegative matrices, parallels the developed theory for positive semidefinite matrices \cite{FJ}. 

Multiplicative determinantal inequalities may also be interpreted via ratios where both the numerator and the denominator consist of products of minors. The multiplicative inequalities involving products of pairs of minors (which we call as the smallest multiplicative inequalities) have been completely characterized for totally nonnegative matrices due to the work of Fallat--Gekhtman--Johnson \cite{FGJ03} and Skandera \cite{Ska04} {(and later further studied by Rhoades--Skandera \cite{RS2005, RS2006})}, and studies involving the generators of such ratios have also been investigated in \cite{BF08}. In this work we introduce a tool, which is a set theoretic operation, that can be applied on determinantal inequalities for totally nonnegative matrices to potentially deduce other families of inequalities for these matrices (see Section~\ref{section-row-op}). Using this set operation, we explore the theory of multiplicative determinantal inequalities, and obtain a novel classification of the multiplicative inequalities characterized due to Fallat--Gekhtman--Johnson \cite{FGJ03} and Skandera \cite{Ska04} {(which was re-derived by Rhoades--Skandera \cite{RS2005, RS2006})}. Moreover, we thoroughly describe these multiplicative inequalities by unraveling a disguised remarkable association among these inequalities in Sections~\ref{Section-multi-results} and \ref{Multi-proofs}.

Additive inequalities for totally nonnegative matrices are real linear combinations of products of minors that are nonnegative over all totally nonnegative matrices. These inequalities are far less understood and studied. (See \cite{Skan-Sos} for a very recent advance along these lines.) For the most part, such determinantal expressions and inequalities stem from the classical Pl\"{u}cker and the Laplace identities of the determinant, both of which are known to hold for all matrices \cite{HJ1, MM}. Perhaps the most prominent and utilized such expression in the context of totally nonnegative matrices is the short-term (or 3-term) Pl\"{u}cker identity, which is a key ingredient used in proving that invertible totally nonnegative matrices are dense in all totally nonnegative matrices \cite{A, K2}. The short-term Pl\"{u}cker identity also arises in the verification of the Whitney-type (or Neville) eliminations on totally nonnegative matrices. This leads to the principal fact that these matrices exhibit an elementary bidiagonal factorization \cite{Fal01, W}. Considering the Laplace expansion along the first row of totally nonnegative matrices, Gantmacher--Krein \cite{GK2} observed a telescopic series of inequalities for these matrices. Using the index-row operations that we introduce in Section~\ref{section-row-op}, we completely refine the classical Laplace and Karlin's identity along the lines of Gantmacher--Krein \cite{GK2}, and seek another novel set of fluctuating inequalities.

We acknowledge the intimate relationship between expressions involving minors of totally nonnegative matrices and relations involving Pl\"{u}cker coordinates corresponding to certain cluster variables associated to a given Cluster algebra. (For example, such connections can be found in \cite{MS, MSp}.) We have decided to circumvent relaying on this connection and remain within the matrix theory setting where this analysis applies nicely and seems to exhibit the most natural applications. Furthermore, it seems that the index-row operation introduced in this work may not have an explicit operation that is known to preserve the positive Grassmannian. Such implications and related consequences in the more algebraic setting will be put aside for possible future work along these lines.

In Section~\ref{Section-add-ineq}, we introduce our main results on additive inequalities and systematically go through classical identities and inequalities due to Laplace \cite{L}, Karlin, M\"{u}hlbach, Gasca \cite{K2, MG, M90}, Gantmacher--Krein \cite{GK2}. Other historically significant relationships that also arise can be found in \cite{Cau, Cay, L, M}. Then, in Section~\ref{Section-TN}, we briefly discuss planar networks that specifically rooted from bidiagonal factorizations of totally nonnegative matrices. Then, we devote Section~\ref{section-row-op} to introduce and establish the functionality of the index-row/column operations, which involves proving our key results that capture the essence of the overall aim in this paper. In Section~\ref{Proofof-Section-add-ineq} we prove the additive inequalities discussed in Section~\ref{Section-add-ineq} with the assistance of the index operations introduced in Section~\ref{section-row-op} and the preliminary discussion in Section~\ref{Section-TN}. Then, in Sections~\ref{Section-multi-results} and \ref{Multi-proofs} we discuss the multiplicative inequalities, which includes a thorough discussion on the works of Fallat--Gekhtman--Johnson \cite{FGJ03} and Skandera \cite{Ska04} {(and later revisited in \cite{RS2005, RS2006})}. Finally in Section~\ref{Section-conclusion}, we conclude this paper with a glimpse of the action of the index operations on the Barrett--Johnson determinantal inequality by Skandera--Soskin \cite{Skan-Sos}.

\section{Additive inequalities: main results}\label{Section-add-ineq}

Laplace \cite{L} proved that the product of a real square matrix $A$ with its adjugate $\adj A$ is a scalar matrix. More precisely:
\begin{theorem}[Laplace \cite{L}]\label{Laplace-thm}
Suppose $n\geq 1$ is an integer, and $A:=(a_{ij})$ is an $n\times n$ real matrix. Then
\begin{eqnarray}\label{Laplace-iden}
\sum_{k=1}^{n}(-1)^{j+k}a_{ik}\det A_{jk} =
\begin{cases}
 \det A & \mbox{ if } i=j, \mbox{ and}\\ 
 0 & \mbox{ otherwise,}
\end{cases}
\end{eqnarray}
where $A_{jk}$ is the submatrix of $A$ obtained upon omitting $j^{\mbox{th}}$ row and $k^{\mbox{th}}$ column.
\end{theorem}

It is natural to seek a refinement of the Laplace identity \eqref{Laplace-iden}, possibly for some distinguished class of matrices. One class that has prominently featured in many areas of mathematics consists of the totally negative matrices -- see the opening paragraph. Schoenberg showed \cite{Sc} that every such matrix $A$ has the variation diminishing property: the number of sign-changes in the coordinates of $Ax$ is at most the number of sign-changes in $x$, for every vector $x$. A key ingredient in the proof of this property (see e.g. \cite{FJb, Pi}) involves manipulating the Laplace expansion. Another fundamental result that involves manipulating the Laplace expansion is the bidiagonal factorization theorem or the classification theorem due to Whitney \cite{W} (see Theorem~\ref{TN-classification}). Therefore, we seek to refine the Laplace identity \eqref{Laplace-iden}, for totally nonnegative matrices. Let us formally define totally nonnegative matrices before we continue this discussion.

\begin{defn}[Totally nonnegative matrices]
Suppose $m,n\geq 1$ are integers, and let $A$ be an $m\times n$ real matrix. Then $A$ is called totally nonnegative (TN) provided for every square submatrices $B$ of $A,$ $\det B \geq 0.$
\end{defn}

Now suppose we refer to the left hand side of the Laplace identity \eqref{Laplace-iden} as the Laplace expansion along the $i^{\mbox{th}}$ row (of $A$) and $j^{\mbox{th}}$ column (of $\adj A).$ Then \eqref{Laplace-iden} shows that this Laplace expansion yields the determinant of $A$ if $i=j,$ otherwise it returns zero. Gantmacher--Krein \cite{GK2} took the first step in the refinement of Laplace identity, and proved the following set of inequalities for $i=j=1$.

\begin{theorem}[Gantmacher--Krein \cite{GK2}]\label{GK-thm}
Suppose $n,l\geq 1$ are integers such that $l\leq n.$ If $A:=(a_{ij})$ is an $n\times n$ totally nonnegative matrix then,
\begin{equation}\label{GK-ineq}
\sum_{k=1}^{l}(-1)^{1+k}a_{1k}\det A_{1k} 
\begin{cases}
\geq \det A & \mbox{ if }l \mbox{ is odd, and} \\
\leq \det A & \mbox{ otherwise,} 
\end{cases}
\end{equation}
where $A_{1k}$ is obtained upon omitting the $1^{\mbox{th}}$ row and $k^{\mbox{th}}$ column in $A.$ 
\end{theorem}

Gantmacher--Krein \cite{GK2} proved that the \textit{partial} Laplace expansions along the first row of TN matrices form a sequence of fluctuating inequalities about the determinant, where $\det A$ happens to be on the right hand side when $i=j=1$ in \eqref{Laplace-iden}. We say that these fluctuate about the determinant as the inequality shifts from being $\geq$ to $\leq,$ or vice versa, as $l$ takes successive values in $\{1,2,\ldots,n\}.$ We shall use the same term in similar contexts later as well. Now, to see how we refine \eqref{GK-ineq} for all other values of $i$ and $j,$ consider the following for $n\times n$ TN matrices $A:$
\begin{equation}\label{matrix1-GK}
A^*:=(a_{ik}^*):= A \circ (\adj A)^T -(\det A)\I,
\end{equation}
where $\circ$ represents the entrywise product of matrices. Theorem~\ref{GK-thm} is concerned with the partial row sums along the first row of $A^*,$ i.e., Gantmacher--Krein inequalities~\eqref{GK-ineq} is the same as, 
\begin{align}\label{matrix1-GK1}
(-1)^{1+l}\sum_{k=1}^{l} a_{1k}^*\geq 0, \mbox{ for all }l\in \{1,\ldots,n\}.
\end{align}

Since $A':=(a_{n+1-j,n+1-k})$ is also TN, whenever $A=(a_{jk})$ is, it would be easy to see that \eqref{matrix1-GK1} holds along the last row of $A^*$ as well. However, it is not immediate how this extends to the other rows in $A^*.$ We provide this extension which refines the Laplace identity \eqref{Laplace-iden} for all $i=j$ cases over totally nonnegative matrices. On the other hand, the Laplace identity~\eqref{Laplace-iden} for all $i\neq j$ cases can be obtained by using index-row operations $\Ro_{(u,v)}.$ We introduce these in Section~\ref{section-row-op} and see that $\Ro_{(u,v)}$ are operations that act on inequalities and produce valid inequalities over TN matrices. Below we state the complete refinement.


\begin{utheorem}\label{Add-thm-1}
Let $n,l\geq 1$ be integers, and suppose $A:=(a_{ij})$ is an $n\times n$ totally nonnegative matrix. Then the following inequalities hold: 
\begin{itemize}
\item[$(1)$] If $i\in [n],$ and $a_{ik}^*$ is defined as in \eqref{matrix1-GK} for $A,$ then
\begin{eqnarray}\label{Add-thm-1-ineq-1}
(-1)^{i+l}\sum_{k=1}^{l}a_{ik}^* \geq 0, \mbox{ for all }l\in \{1,\ldots,n\}.
\end{eqnarray}  

\item[$(2)$] If $i,j\in [n]$ such that $i\neq j,$ then
\begin{eqnarray}\label{Add-thm-1-ineq-2}
(-1)^{j+l}\sum_{k=1}^{l}(-1)^{j+k}a_{ik}\det A_{jk} \geq 0, \mbox{ for all }l\in \{1,\ldots,n\},
\end{eqnarray}
where $A_{jk}$ is obtained upon omitting the $j^{\mbox{th}}$ row and $k^{\mbox{th}}$ column in $A.$ 
\end{itemize}
\end{utheorem}

\begin{remark}[Hierarchy among inequalities]\label{Hierarchy-rem-0}
Note that, Theorem~\ref{Add-thm-1}(1) for $i=1$ specializes to the Gantmacher--Krein inequality~\eqref{GK-ineq}. Moreover, as explained later (in Section~\ref{Proofof-Section-add-ineq}), these index-row operations $\Ro_{(u,v)}$ help derive Theorem~\ref{Add-thm-1}(2) for all $i \neq j$ from Theorem~\ref{Add-thm-1}(1) for any $i'$ -- but not in the reverse direction. (In fact, any $i \neq j$ can be derived from any $i'\neq j'$ in (2) -- see Lemma~\ref{Main-lemma-3}(2).) Thus, these index-row (and index-column) operations reveal a ``hierarchy'' among these additive inequalities for TN matrices, in which the inequalities in part (1) of Theorem \ref{Add-thm-1} are ``higher'' than those in part (2). (We will recall these in Remarks~\ref{Hierarchy-remark} and \ref{FGJ-Skan-comp-remark}.)
\end{remark}

Theorem~\ref{Add-thm-1} precisely articulates that the partial sums on the left hand side of the Laplace identity \eqref{Laplace-iden} \textit{subtly} fluctuate about the corresponding right hand side over all TN matrices. This can be seen in further generality by observing that the Laplace identity \eqref{Laplace-iden} for $i\neq j$ are special cases of the classical Karlin's identity \eqref{Karlin-iden}. Since there is a refinement of Laplace identity in terms of Theorem~\ref{Add-thm-1}, it is natural to seek the refinement of Karlin's identity along similar lines. We list a few notations to continue further:

\begin{defn}\label{Defn-Karlin} Here onwards we shall adopt the following notation:
\begin{itemize}
\item[$(a)$] Suppose $m,n\geq 0$ are integers such that $m\leq n,$ then define 
\begin{align*}
[m,n]:=\{m,\ldots,n\},~ [n]&:=\{1,\ldots,n\} \mbox{ whenever }n\geq 1, \mbox{ and } [0]:=\{0\}.
\end{align*}
\item[$(b)$] Suppose $n\geq 1$ is an integer. Then for any set $S\subseteq [n],$ define 
\begin{align*}
S^{\mathsf{c}}:=[n]\setminus S, \mbox{ and use }|S| \mbox{ to denote the number of elements in }S.
\end{align*} 
\item[$(c)$] Suppose $n\geq 1$ be an integer. For all $S,T\subseteq [n]$ and any $n\times n$ matrix $A,$ we let $A\big{(}S\big{|}T\big{)}$ denote the submatrix with row indices in $S$ and column indices in $T$. We further use $A\big{(}S\big{)}$ to denote $A\big{(}S\big{|}S\big{)},$ and if $S$ is empty, then $\det A \big{(}S\big{)}:=1.$
\end{itemize}
\end{defn}

Now we state Karlin's identity:

\begin{theorem}[Karlin \cite{K2},~M\"{u}hlbach \cite{M90}]\label{Karlin-thm}
Let $n\geq 1$ be an integer. Suppose $T$ is a subset of $[n]$ and $V:=\{v_1<\cdots<v_m \}:=[n]\setminus T.$ Then for all $p \in [n]$ and set $S,$ such that $S\subseteq [n] \setminus \{p\}$ and $|S|=|T|+1,$
\begin{equation}\label{Karlin-iden}
\sum_{k=1}^{m} (-1)^{1+k} \det A\big{(}S\big{|}T\cup \{v_k\}\big{)} \det A\big{(}[n]\setminus\{p\}\big{|}[n]\setminus\{v_k\}\big{)}=0,
\end{equation}
for all $n\times n$ real matrices $A.$
\end{theorem}

As observed previously, Karlin's identity \eqref{Karlin-iden} encodes the Laplace identity \eqref{Laplace-iden} for all $i\neq j$. This can be seen by taking $S=\{i\},$ $T=\emptyset,$ and $p=j,$ in Theorem~\ref{Karlin-thm}. This attribute of Karlin's identity \eqref{Karlin-iden} and Theorem~\ref{Add-thm-1}\eqref{Add-thm-1-ineq-1} intrigued us to investigate if, like the Laplace identity~\eqref{Laplace-iden}, Karlin's identity accommodates sequence of inequalities as well over TN matrices. Using the preliminary discussion in Sections~\ref{Section-TN} and novel ideas in~\ref{section-row-op}, we found this speculation to be correct, and below is the precise result.

\begin{utheorem}\label{Add-thm-2}
Let $n\geq 1$ be an integer. Suppose $T$ is a subset of $[n]$ and $V:=\{v_1<\cdots<v_m \}:=[n]\setminus T.$ Then for all $p\in [n]$ and set $S,$ such that $S\subseteq [n] \setminus \{p\}$ and $|S|=|T|+1,$
\begin{equation}\label{Add-thm-2-ineq}
(-1)^{1+l}\sum_{k=1}^{l} (-1)^{1+k} \det A\big{(}S\big{|}T\cup \{v_k\}\big{)} \det A\big{(}[n]\setminus\{p\}\big{|}[n]\setminus\{v_k\}\big{)}\geq 0,
\end{equation}
for all $l\in [m],$ and all $n\times n$ totally nonnegative matrices $A.$
\end{utheorem}


Karlin's identity \eqref{Karlin-iden} refines into a sequence of inequalities that fluctuates about zero for totally nonnegative matrices. This beautifully aligns with the original Gantmacher--Krein inequalities \eqref{GK-ineq} and with its refinement Theorem~\ref{Add-thm-1}. These alignments intrigued us to examine further and seek other such fluctuating inequalities. Our next main result provides a novel undiscovered set of such inequalities for totally nonnegative matrices.

\begin{defn}[Disjoint union and a set ordering]
Suppose $n\geq 1$ is an integer and $P,$ $P_1,$ and $P_2$ are subsets of $[n].$ We say that:
\begin{itemize}
\item[$(a)$] $P=P_1\sqcup P_2$ provided $P=P_1\cup P_2$ and $P_1\cap P_2=\emptyset,$ and
\item[$(b)$] $P_1< P_2$ if every element in $P_1$ is \textit{strictly} less than those in $P_2.$
\end{itemize}
\end{defn} 

\begin{utheorem}\label{Add-thm-3}
Let $n,d\geq 1$ be an integer, and let $P,Q\subseteq[n]$ be nonempty with $|P|+|Q|=n.$ Suppose $d=|P|,$ and define the sets
\begin{align*}
J_{dl} := [n-d,n]\setminus\{n-d+l\}, \mbox{ for all }l\in [0,d].
\end{align*}
Suppose $P=P_{1}\sqcup P_{2}$ and $Q=Q_1\sqcup Q_2$ such that $P_1 < P_2$ and $Q_1 < Q_2.$ If $Q_1 \subseteq P_1$ and $P_2 \subseteq Q_2,$ then
\begin{align}\label{Add-thm-3-ineq}
(-1)^{1+l}\sum_{k=0}^{l} (-1)^{1+k} \det A\big{(}P\big{|}J_{dk}\big{)} \det A\big{(}Q\big{|}[n]\setminus J_{dk}\big{)} \geq 0,
\end{align}
for all $l\in [0,d],$ and all $n\times n$ totally nonnegative matrices $A.$
\end{utheorem}

Theorem~\ref{Add-thm-3} is an attempt to refine the generalized Laplace identity (see \cite{M}) along the lines of Gantmacher--Krein inequality \eqref{GK-ineq}. The ensuing proof uses the identification of totally nonnegative matrices with planar networks to prove it first for $P=Q^{\mathsf{c}}=[1,d].$ Then to obtain the remaining inequalities corresponding to other values of $P$ and $Q,$ we make use of the inequality preserving index-row operations $\Ro_{(u,v)}$ (which we introduce in Section~\ref{section-row-op}). This, once again, gives us that hierarchy that we discussed in Remark~\ref{Hierarchy-rem-0}.

\section{Preliminaries: total nonnegativity and planar networks}\label{Section-TN}

Let $n\geq 1$ be an integer, and $u,v\in [n].$ We use $\I$ to denote the $n\times n$ identity matrix, and $\El_{uv}$ to refer to the basic $n\times n$ matrix with $1$ at $(u,v)$ position, and zero otherwise. We refer to $n\times n$ matrices of the form $\I+w\El_{uv},$ for $w\geq 0,$ as elementary bidiagonal provided $u$ and $v$ are consecutive. Using these elementary bidiagonal matrices, we can state the classification of square totally nonnegative matrices which we frequently use in this paper (see also \cite{Fal01, FZ2}):

\begin{theorem}[Whitney \cite{W}]\label{TN-classification}
A nonsingular totally nonnegative matrix $A$ can be factorized as $DE,$ where  $D$ is a diagonal matrix with positive diagonals, and $E$ factors into finitely many elementary bidiagonal matrices $\I+w\El_{uv},$ where $w\geq 0$ and $u,v$ are consecutive integers. More precisely, let $n\geq 1$ be an integer and $A$ an $n\times n$ nonsingular totally nonnegative matrix. Then
\begin{align}\label{TN-classification-eqn}
A=\prod_{j=1}^{n-1} \prod_{k=n-1}^{j} \big{(} \I+w_{j,k}\El_{k+1,k} \big{)}~D \prod_{j=n-1}^{1} \prod_{k=n-1}^{j} \big{(} \I+ w'_{j,k+1}\El_{k,k+1} \big{)}
\end{align}
where $D$ is a diagonal matrix with positive diagonals, and it can be made to appear at either end of the factorization or between any of the elementary bidiagonal factors; the factors are multiplied in the given order of $j$ and $k,$ and each $w_{j,k},w'_{j,k+1}\geq 0.$ In addition to this, the set of $n\times n$ nonsingular totally nonnegative matrices is dense in the set of $n\times n$ totally nonnegative matrices.
\end{theorem}

Based on this classification, totally nonnegative matrices can be recognized as weighted planar networks (in Figures~\ref{Smaller figure} and \ref{Planar-bw}).

\begin{defn}[Planar network and weighted sums]\label{PlanarNet-defn}
Let $n\geq 1$ be an integer.
\begin{itemize}
\item[$(1)$] A weighted planar network of order $n$ is an acyclic directed planar graph $G=(V,E),$ with nonnegative weights on the edges, from source vertices $S$ to sink vertices $T$ with $|S|=|T|=n,$ such that the planar diagram can be realised ensuring that all the edges are line segments.
\item[$(2)$] Suppose $G=(V,E)$ is a weighted planar network of order $n$ with source $S$ and sink $T,$ each labelled $1,\ldots , n.$ For all $I,J\subseteq [n]$ such that $|I|=|J|,$ define
\begin{equation}\label{PlanarNet-Matrix-defn}
w_{IJ}:=\sum_{P:I\to J} w_{P},
\end{equation}
where the sum is taken over the collection of family of disjoint paths $P$ from vertices $I$ in source to vertices $J$ in sink, and where $w_P$ is the product of edge weights in the family of disjoint paths $P.$
\end{itemize}
\end{defn}

\begin{figure}[ht]
    \includegraphics[width=10.0cm, height=5cm]{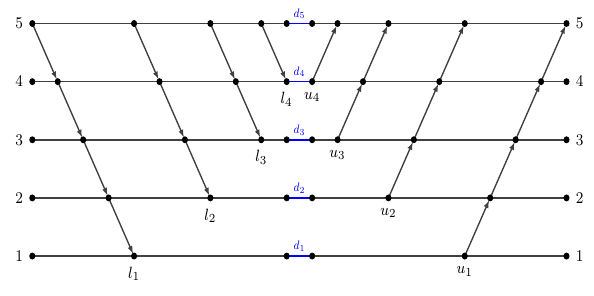}
    \caption{A planar network representing $5\times 5$ totally nonnegative matrices, where every edge has nonnegative weight and is directed from left to right.}
    \label{Smaller figure}
\end{figure}

The identification of totally nonnegative matrices and the weighted planar networks is established using the following:

\begin{lemma}[Gessel--Viennot \cite{GV}, Lindstr\"{o}m \cite{Lin73}]\label{Lin-lemma}
Let $n\geq 1$ be an integer, and let $G:=(V,E)$ be a weighted planar network of order $n.$ Suppose $A:=(a_{ij})$ is an $n\times n$ matrix such that $a_{ij}:=w_{\{i\}\{j\}},$ as in~\eqref{PlanarNet-Matrix-defn}. Then, minors of $A$ corresponding to row indices $I$ and column indices $J$ are given by 
$
\det A\big{(}I\big{|}J\big{)} = w_{IJ}. 
$
In particular, $A$ is totally nonnegative.
\end{lemma}

The planar networks in Figures~\ref{Smaller figure} and \ref{Planar-bw} correspond respectively to all $5\times 5$ and $n\times n$ TN matrices, and are constructed using Lindstr\"{o}m's Lemma~\ref{Lin-lemma} and Whitney's classification Theorem~\ref{TN-classification}. We briefly describe this in the subsection below:

\subsection*{Construction of planar networks}\label{Construction-TN}
Suppose $n\geq 1$ is a nonnegative integer, and $u,v\in [n]$ consecutive integers. Now consider the elementary bidiagonal factor $\I+w\E_{uv},$ with $w>0,$ mentioned in Theorem~\ref{TN-classification}. This factor is TN and the corresponding planar network is a graph on $2n$ vertices, which forms the source vertices labelled $1,2,\ldots,n,$ and the sink vertices labelled $1,2,\ldots,n.$ The directed edges are tuples with the first component being a source vertex and the second component being a sink vertex. These edges form the set 
$$
\big{\{}(k,k):k\in [n]\big{\}}\cup\big{\{}(u,v)\big{\}}, \mbox{ where edges }(k,k) \mbox{ weigh } 1, \mbox{ and edge } (u,v) \mbox{ weighs }w.
$$
Similarly, the planar network for a diagonal matrix $D:=\diag(d_1,\cdots,d_n),$ with each $d_{k} > 0,$ is a directed graph on $2n$ vertices with the source and sink vertices labelled $1,2,\ldots,n$ each. The corresponding directed edges are 
$$
\big{\{}(k,k):k\in [n]\big{\}}, \mbox{ where the edge } (k,k) \mbox{ weighs } d_k \mbox{ for all }k\in \{1,2,\ldots,n\}.
$$

\begin{figure}[ht]
    \includegraphics[width=14.5cm, height=10.2cm]{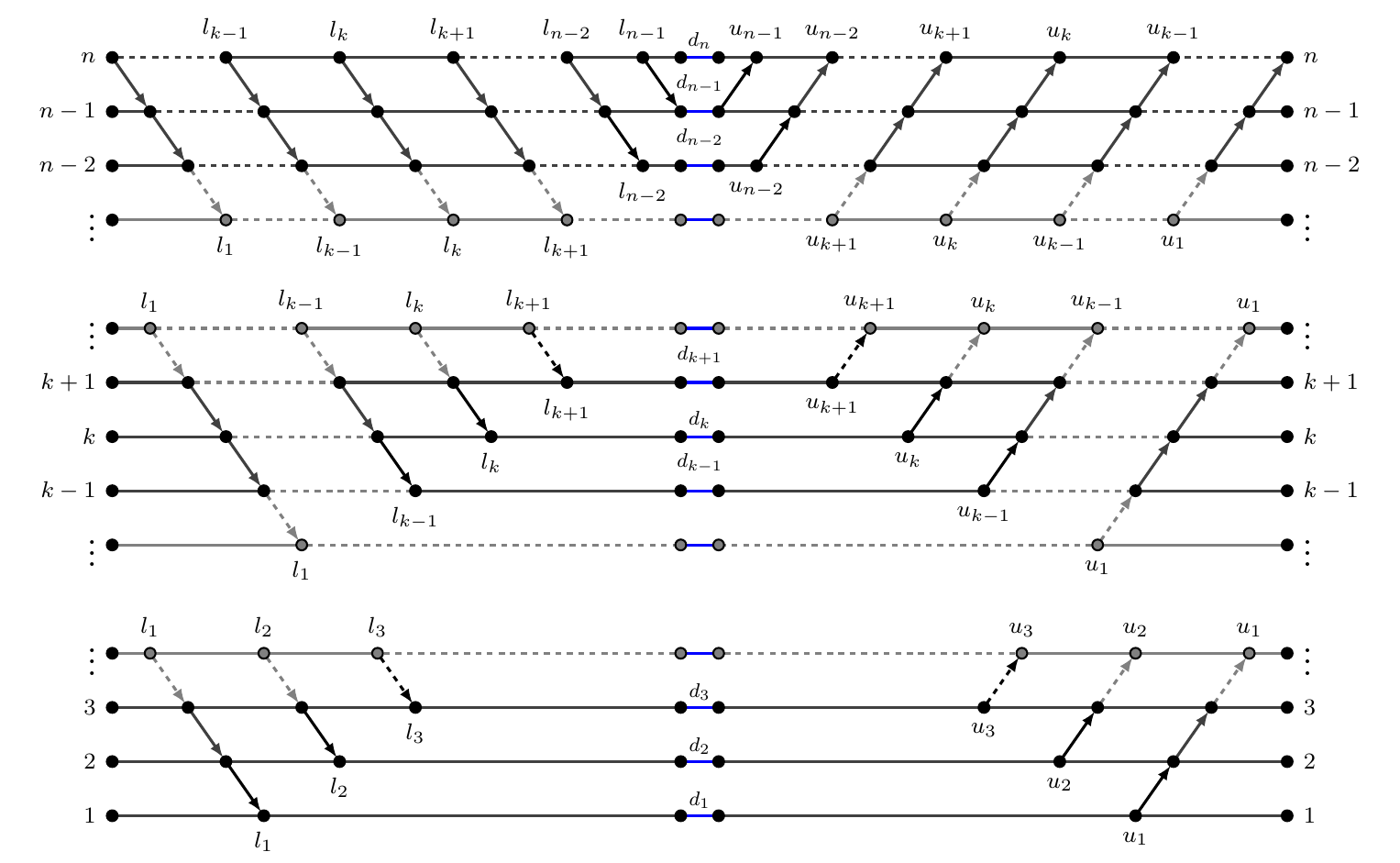}
    \caption{A planar network representing $n\times n$ totally nonnegative matrices, where every edge has nonnegative weight and is directed from left to right.}
    \label{Planar-bw}
\end{figure}

The Cauchy--Binet formula (see \cite{HJ1}) implies that the planar network corresponding to any nonsingular $n\times n$ TN matrix $A$ is the concatenation of planar networks for elementary bidiagonal factors $\I+w\E_{uv}$ and diagonal matrices $D$ that appears in the bidiagonal factorization \eqref{TN-classification-eqn} of $A$ (see \cite{Fal01, FZ2}). To briefly elaborate on this:

\begin{itemize}
\item[$(a)$] Draw the planar networks for each of those factors.
\item[$(b)$] Then identify the sink vertices of each of the factors with the corresponding source vertices of the factor on its immediate right. For instance, connect sink vertex $k$ of the leftmost factor with the source vertex $k$ of the second leftmost factor, and the sink vertex $k$ of the second leftmost factor with the source vertex $k$ of the third leftmost factor, and so on, via directed edges with weight $1,$ for each $k.$
\item[$(c)$] To avoid redundancy, mention only the vertices that are at the intersection of a horizontal and a non-horizontal edge, as in Figures~\ref{Smaller figure} and~\ref{Planar-bw}.
\end{itemize}

Therefore, the weighted planar networks in Figure~\ref{Smaller figure} and \ref{Planar-bw} identifies respectively with all $5\times 5$ and $n\times n$ nonsingular TN matrices. The continuity of determinant and the density of nonsingular TN matrices ensure that it is enough to prove a determinantal identity or inequality for nonsingular TN matrices. We shall come back to Figure~\ref{Planar-bw} in Section~\ref{Proofof-Section-add-ineq} (for the proof of Theorem~\ref{Add-thm-3}) where we will describe and utilize various details in the network.

\begin{unotion}\label{main-idea}
This notion will frequently be used in this work, and can be proved using the well known Cauchy--Binet formula (see \cite{HJ1}). Suppose $n\geq 1,$ and $A$ is an $n\times n$ TN matrix. Let $w>0$ and $u,v\in [n]$ be consecutive integers, and define $B:=(\I+w\E_{uv})A.$ Then for all $I,J\subseteq [n]$ with $|I|=|J|,$ 
\begin{equation}\label{main-idea1}
\det B \big{(}I\big{|}J\big{)} = 
\begin{cases}
\det A \big{(}I\big{|}J\big{)} + w \det A \big{(}I(u,v)\big{|}J\big{)} & \mbox{ if } u\in I \mbox{ and } v\not\in I, \mbox{ and}\\
\det A\big{(}I\big{|}J\big{)} & \mbox{ otherwise},\\
\end{cases}
\end{equation}
where $
I(u,v):=
\begin{cases}
(I\setminus\{u\})\cup\{v\} & \mbox{ if }u\in I \mbox{ and } v\not\in I,\\
I & \mbox{ otherwise}.
\end{cases}
$ 

Similar would be the result if $B:=A(\I+w\E_{uv}).$ We leave the proof of these to the reader.
\end{unotion}

\section{Index-row/column operations preserving inequalities}\label{section-row-op}

A generic determinantal inequality for totally nonnegative matrices $A:=(a_{ij})$ is a linear combination of products of finitely many minors of the form $\det A\big{(}I\big{|}J\big{)}$ followed by a inequality sign $\geq$ and zero. This inequality can be altered to obtain another inequality over TN matrices by replacing each of the minors with $\det A\big{(}J\big{|}I\big{)}$ since the underlying matrix $A^T:=(a_{ji})$ is TN whenever $A$ is. Similarly we get another inequality by replacing each $\det A\big{(}I\big{|}J\big{)}$ with $\det A\big{(}n+1-I\big{|}n+1-J\big{)}$ (see Definition~\ref{cce-defn}$(4)$ for $n+1-I$ and $n+1-J$) since $A':=(a_{n+1-i,n+1-j})$ is TN whenever $A$ is. There are also other set theoretic operations which alter the row and column indices of the minors which naturally reshape into operations on inequalities that preserve the inequalities. These are discussed in \cite{FGJ03}, and are called {complement}, {shift}, {insertion}, and {deletion}. 

These operations have a very limited reach in obtaining novel inequalities from the designated known ones. For example, the additive inequalities in Section~\ref{Section-add-ineq} can not be obtained using any of the aforementioned/known set operations on any known inequality. This is essentially because none of these operations are meant to deal with row and column indices that are not common among all minors in the inequalities. Having said this, we should also acknowledge that a larger set of inequalities for totally nonnegative matrices would certainly make the discussion around these matrices more frictionless. The introduction of novel set operations is our attempt in that direction, while it also enables us to remove the limitations of the known set theoretic operations and provides further applications.

Considering that a generic determinantal inequality is a real linear combination of products of finitely many minors, we begin with defining a set operation that operates on ordered family of ordered collection of sets.

\begin{defn}[Set operation $\So_{(u,v)}$]\label{set-op}
Suppose $n\geq 1$ be an integers and $I\subseteq [n],$ and define for all consecutive $u,v\in [n]$:
$$
I(u,v):=
\begin{cases}
(I\setminus\{u\})\cup\{v\} & \mbox{ if }u\in I \mbox{ and } v\not\in I,\\
I & \mbox{ otherwise}.
\end{cases}
$$
Now, let $\mathcal{I}$ be an ordered nonnempty finite indexing set, and $\gamma:\mathcal{I} \to \Z_{>0}.$ Suppose $I_{ij}\subseteq [n]$ for all $i\in \mathcal{I}$ and $j\in [\gamma(i)].$ The set operation $\So_{(u,v)}$ is defined for all consecutive $u,v\in [n]$: 
\begin{align*}
\So_{(u,v)}\Big{(}\big{(}I_{i1},\ldots,I_{i\gamma(i)}\big{)}:i\in \mathcal{I}\Big{)}:=\Big{(}\big{(}I_{i1}(u,v),\ldots,I_{i\gamma(i)}(u,v)\big{)}:i\in \mathcal{I}\Big{)}.
\end{align*}
\end{defn} 

The set operation $\So_{(u,v)}$ can be seen as a refined form of the shift operation \cite{FGJ03}. We will use $\So_{(u,v)}$ as a set theoretic operation to mainly express ideas in proofs and theorems. For more profound consequences, and to apply upon determinantal expressions and inequalities, we define a refined form of $\So_{(u,v)}:$

\begin{defn}[Index-row operation $\Ro_{(u,v)}$]\label{row-op}
Suppose $n\geq 1$ is an integer, and let $\mathcal{I}$ be an ordered nonnempty finite indexing set, and $\gamma:\mathcal{I} \to \Z_{>0}.$ Suppose $I_{ij}\subseteq [n]$ for all $i\in \mathcal{I}$ and $j\in [\gamma(i)].$ The index-row operation $\Ro_{(u,v)}$ is defined for all consecutive $u,v\in [n]$:
\begin{align*}
\Ro_{(u,v)}\Big{(}\big{(}I_{i1},\ldots,I_{i\gamma(i)}\big{)}:i\in \mathcal{I}\Big{)}:=\Big{(}\big{(}I_{i1}(u,v),\ldots,I_{i\gamma(i)}(u,v)\big{)}:i\in \mathcal{I}(u,v)\Big{)},
\end{align*}
where,
\begin{align*}
i(u,v)&:=\big{\{}j\in [\gamma(i)]: u \in I_{ij} \mbox{ and } v \not \in I_{ij} \big{\}}, \mbox{ and }\\
\mathcal{I}(u,v)&:= \big{\{}i \in \mathcal{I}: |i(u,v)|= \max_{i\in \mathcal{I}} |i(u,v)| \big{\}}.
\end{align*}
We also define, (for Lemma~\ref{Main-lemma-2})
\begin{align*}
\mathcal{I}(u,v,k)&:=\big{\{}i \in \mathcal{I}: |i(u,v)|=k \big{\}} \mbox{ for } k\in \Z_{\geq 0}.
\end{align*}
\end{defn}

We are calling $\Ro_{(u,v)}$ an index-row operation (or just a row operation) because it can be seen as operating on the row index sets in the determinantal expressions and inequalities. When $\Ro_{(u,v)}$ acts on an inequality over TN matrices, then it produces a determinantal expression which forms a valid inequality. The following result validates this claim:

\begin{ulemma}\label{Main-lemma-1}
Suppose $n\geq 1$ is an integer, and let $\mathcal{I}$ be a nonempty indexing set, and $\gamma:\mathcal{I} \to \Z_{>0}.$ Suppose $I_{ij},J_{ij}\subseteq [n],$ such that $|I_{ij}|=|J_{ij}|,$ for $i\in \mathcal{I}$ and $j\in [\gamma(i)].$ Then for real $\lambda_{i},~i \in \mathcal{I},$ assertion $(1)$ implies $(2)$:
\begin{itemize}
\item[$(1)$] For all $n\times n$ totally nonnegative matrices $A,$
$$F_{\mathcal{I}}(A):=\sum_{i\in \mathcal{I}} \lambda_{i} \prod_{j=1}^{\gamma(i)} \det A\big{(}I_{ij}\big{|}J_{ij}\big{)}\geq 0.$$
\item[$(2)$] For all $n\times n$ totally nonnegative matrices $A,$  
\begin{align*}
G_{\mathcal{I}(u,v)}(A):=\sum_{i\in \mathcal{I} (u,v)} \lambda_{i} \prod_{j=1}^{\gamma(i)} \det A\big{(}I_{ij}(u,v)\big{|}J_{ij}\big{)}\geq 0,
\end{align*}
for all consecutive integers $u,v\in [n].$
\end{itemize}
This is under the convention that the sum and the product over empty set respectively are $0$ and $1.$ Moreover, a similar result is true where the column indices $J_{ij}$ are altered to $J_{ij}(u,v)$ in assertion $(2)$ (see Remark~\ref{Column-op}).  
\end{ulemma}

The operation $\Ro_{(u,v)}$ can be seen as a a set operation acting on the row indices of any inequality for TN matrices and yielding inequalities for TN matrices:
\begin{align*}
\sum_{i\in \mathcal{I}} \lambda_{i} \prod_{j=1}^{\gamma(i)} \det A\big{(}I_{ij}\big{|}J_{ij}\big{)}\geq 0~~\xRightarrow[]{\Ro_{(u,v)}}~~\sum_{i\in \mathcal{I}(u,v)} \lambda_{i} \prod_{j=1}^{\gamma(i)} \det A\big{(}I_{ij}(u,v)\big{|}J_{ij}\big{)}\geq 0,
\end{align*}
where each determinantal expressions forms an inequality over all $n\times n$ TN matrices. For efficiency we indicate this transformation by simply writing,
\begin{align*}
\Ro_{(u,v)}\Big{(}\big{(}I_{i1},\ldots,I_{i\gamma(i)}\big{)}:i\in \mathcal{I}\Big{)}=\Big{(}\big{(}I_{i1}(u,v),\ldots,I_{i\gamma(i)}(u,v)\big{)}:i\in \mathcal{I}(u,v)\Big{)},
\end{align*}
where notations above are as in Definition~\ref{row-op}. 

\begin{remark}[Significance of Lemma~\ref{Main-lemma-1}, and its ``converse'']\label{MainLemma-converse}
The discussion following Lemma~\ref{Main-lemma-1} so far implies that any composition of the index-row operations $\Ro_{(u,v)}$ applied on inequalities over TN matrices returns inequalities for TN matrices. This notion will be applicable in almost all the results in this paper (see Sections~\ref{Section-add-ineq} and \ref{Section-multi-results}). On the other hand, it is also natural to ask for the converse of this useful idea. We avoid conjecturing this, but the intended converse -- which can be seen as the converse or ``$(2)$ implies $(1)$'' in Lemma~\ref{Main-lemma-1} -- seems too good to be true in general. However, for the class of all the smallest multiplicative inequalities (see Section~\ref{Section-multi-results}) this converse is indeed true, and is stated in Theorem~\ref{Multi-thm-2}. This novel result provides an entirely different, yet natural outlook at the classification Theorem~\ref{FGJ-Skan-thm}, due to Fallat--Gekhtman--Johnson \cite{FGJ03} and Skandera \cite{Ska04}, for the smallest multiplicative inequalities  (see Remark~\ref{Multi-thm-2-significance} next to Theorem~\ref{Multi-thm-2}).
\end{remark}

\begin{remark}[Index-column operation $\Co_{(u,v)}$]\label{Column-op}
Since the transpose of totally nonnegative matrices is totally nonnegative, index-row operation $\Ro_{(u,v)}$ naturally births index-column operations $\Co_{(u,v)}.$ Therefore, the results that are true for $\Ro_{(u,v)}$ naturally extend to $\Co_{(u,v)}.$ We shall make the relevant remarks regarding the analogous index-column operation $\Co_{(u,v)}$ as needed.
\end{remark}

\begin{remark}[$\Ro_{(u,v)}$ acting on determinantal expressions]\label{MainLemma-action}
Henceforth, we will use the index-row operation $\Ro_{(u,v)},$ for consecutive integers $u,v\geq 1,$ as a set operation acting on the row indices of determinantal expressions (which might include signs like $+,-,\geq,\leq,=,$ etc.) and yielding determinantal expressions. Also, Lemma~\ref{Main-lemma-1} assures that if $\Ro_{(u,v)},$ for consecutive integers $u,v\geq 1,$ acts on a determinantal expression which forms an inequality over TN matrices, then it returns an inequality over TN matrices. The action of $\Ro_{(u,v)}$ on a determinantal expression is governed by Definition~\ref{row-op}, and is summarized below:
\begin{itemize}
\item[$(1)$] Swap row index $u$ with $v$ in all row index sets, whenever \textit{possible,} in the given determinantal expression.
\item[$(2)$] Identify all those product of minors in which this swapping has happened the maximum number of times.
\item[$(3)$] Finally, replace all those product of minors with zero for which the number of those swapping has happened any fewer number of times.
\item[$(4)$] Keep all other math symbols $($like $+,-,\geq,\leq,=,$ etc$.)$ at their original place. Thus we get the resultant determinantal expression. 
\end{itemize} 
For the corresponding remark on the index-column operation $\Co_{(u,v)},$ we replace every ``row'' with ``column'' and every ``$\Ro_{(*,*)}$'' with ``$\Co_{(*,*)}$'' in this remark.
\end{remark}

\begin{defn}[Identity and inverses]\label{inv-op}
Following Definitions~\ref{set-op} and \ref{row-op}, for convenience we define the identity and inverse operations: 
\begin{align*}
\So_{(u,u)}\Big{(}\big{(}I_{i1},\cdots,I_{i\gamma(i)}\big{)}:i\in \mathcal{I}\Big{)}&:=\Big{(}\big{(}I_{i1},\cdots,I_{i\gamma(i)}\big{)}:i\in \mathcal{I}\Big{)},\\
\Ro_{(u,u)}\Big{(}\big{(}I_{i1},\ldots,I_{i\gamma(i)}\big{)}:i\in \mathcal{I}\Big{)}&:=\Big{(}\big{(}I_{i1},\ldots,I_{i\gamma(i)}\big{)}:i\in \mathcal{I}\Big{)},\mbox{ and}\\
\big{(}\Ro_{(u_k,v_k)}\circ \ldots \circ \Ro_{(u_1,v_1)}\big{)}^{-1}&:=\Ro_{(v_1,u_1)}\circ \ldots \circ \Ro_{(v_k,u_k)}
\end{align*}
for all $u,u_k,v_k\in [n]$ and integer $k\geq 1,$ such that each $u_k \in \{v_k-1,v_k,v_k+1\}.$ For the corresponding definitions for the index-column operation $\Co_{(u,v)},$ we replace every ``row'' with ``column'' and every ``$\Ro_{(*,*)}$'' with ``$\Co_{(*,*)}$'' in these definitions. 
\end{defn}

Lemma~\ref{Main-lemma-1} has an analogue for identities over totally nonnegative matrices, and we state and prove these as well. However, we leave the introduction and applicability of the corresponding index-row and index-column operations to future work.

\begin{ulemma}\label{Main-lemma-2}
Suppose $n\geq 1$ is an integer, and let $\mathcal{I}$ be a nonempty indexing set, and $\gamma:\mathcal{I} \to \Z_{>0}.$ Suppose $I_{ij},J_{ij}\subseteq [n],$ such that $|I_{ij}|=|J_{ij}|,$ for $i\in \mathcal{I}$ and $j\in [\gamma(i)].$ Then for real $\lambda_{i},~i \in \mathcal{I},$ assertion $(1)$ implies $(2)$ and $(3)$:
\begin{itemize}
\item[$(1)$] For all $n\times n$ totally nonnegative matrices $A,$
$$F_{\mathcal{I}}(A):=\sum_{i\in \mathcal{I}} \lambda_{i} \prod_{j=1}^{\gamma(i)} \det A\big{(}I_{ij}\big{|}J_{ij}\big{)} = 0.$$

\item[$(2)$] For all $n\times n$ totally nonnegative matrices $A,$  
$$G_{\mathcal{I}(u,v)}(A):=\sum_{i\in \mathcal{I} (u,v)} \lambda_{i} \prod_{j=1}^{\gamma(i)} \det A\big{(}I_{ij}(u,v)\big{|}J_{ij}\big{)}=0,$$
for all consecutive integers $u,v\in [n].$

\item[$(3)$] For all $n\times n$ totally nonnegative matrices $A,$
\begin{align*}
H_{\mathcal{I}(u,v)}^{l}(A):= \sum_{k=l}^{|\mathcal{I}(u,v)|}\sum_{i\in \mathcal{I}(u,v,k)}~\lambda_{i}\sum\limits_{\substack{J' \subseteq i(u,v):\\|J'|=l}} \prod_{j\in J'} \det A\big{(}I_{ij}(u,v)\big{|}J_{ij}\big{)} \prod_{j \in i(u,v)\setminus J'} \det A\big{(}I_{ij}\big{|}J_{ij}\big{)}=0,
\end{align*}
for all $l\in [|\mathcal{I}(u,v)|-1],$ and all consecutive integers $u,v\in [n].$
\end{itemize}
This is under the convention that the sum and the product over empty set respectively are $0$ and $1.$ Moreover, a similar result is true where the column indices $J_{ij}$ are altered to $J_{ij}(u,v)$ in assertions $(2)$ and $(3)$ (see Remark~\ref{Column-op}).
\end{ulemma} 

For the proof of Lemma~\ref{Main-lemma-1} and Lemma~\ref{Main-lemma-2}, suppose $n\geq 1$ is an integer and let $A$ be an $n\times n$ nonsingular totally nonnegative matrix. Let $u,v\in [n]$ be consecutive integers and $w>0$ be arbitrary. Consider the matrix $B:=(\I+w\E_{uv})A,$ where $\E_{uv}$ is the basic matrix. Using notations in Lemma~\ref{Main-lemma-1} and Lemma~\ref{Main-lemma-2}, and using Notion~\ref{main-idea}, observe that $F_{\mathcal{I}}(B)$ is a polynomial in $w$ with constant $F_{\mathcal{I}}(A)$ and leading coefficient $G_{\mathcal{I}(u,v)}(A)$:
\begin{align}\label{Gen-iden-ineq-lemma-equation}
&F_{\mathcal{I}}(B)\nonumber\\
&=\sum_{l=0}^{|\mathcal{I}(u,v)|} w^{l} ~ \sum_{k = l}^{|\mathcal{I}(u,v)|}~\sum_{i\in \mathcal{I}(u,v;k)}~\lambda_{i} \sum\limits_{\substack{J' \subseteq i(u,v):\\|J'|=l}} \prod_{j\in J'} \det A\big{(}I_{ij}(u,v)\big{|}J_{ij}\big{)} \prod_{j \in i(u,v)\setminus J'} \det A\big{(}I_{ij}\big{|}J_{ij}\big{)} \nonumber\\
&= F_{\mathcal{I}}(A) + \sum_{l=1}^{|\mathcal{I}(u,v)|-1} w^{l} ~ \sum_{k = l}^{|\mathcal{I}(u,v)|}~\sum_{i\in \mathcal{I}(u,v;k)}~\lambda_{i}\sum\limits_{\substack{J' \subseteq i(u,v):\\|J'|=l}} \nonumber \\
&\hspace{5cm} \prod_{j\in J'} \det A\big{(}I_{ij}(u,v)\big{|}J_{ij}\big{)} \prod_{j \in i(u,v)\setminus J'} \det A\big{(}I_{ij}\big{|}J_{ij}\big{)} \nonumber \\
&\hspace*{10cm}  +  w^{|\mathcal{I}(p,q)|}~G_{\mathcal{I}(p,q)}(A) \nonumber \\
&= F_{\mathcal{I}}(A) + \sum_{l=1}^{|\mathcal{I}(u,v)|-1} w^{l} H_{\mathcal{I}(u,v)}^{l}(A) + w^{|\mathcal{I}(u,v)|}G_{\mathcal{I}(u,v)}(A).
\end{align}

\begin{proof}[Proof of Lemma~\ref{Main-lemma-1}]
From (\ref{Gen-iden-ineq-lemma-equation}), since $F_{\mathcal{I}}(B) \geq 0$ for all $w>0,$ we have $G_{\mathcal{I}(u,v)}(A)\geq 0$ for all nonsingular TN matrices $A.$ This completes the proof for nonsingular TN matrices. To complete the proof for all TN matrices we use the density of nonsingular TN matrices in TN matrices and the continuity of the determinant.
\end{proof}

\begin{proof}[Proof of Lemma~\ref{Main-lemma-2}]
Once again, from (\ref{Gen-iden-ineq-lemma-equation}), since $F_{\mathcal{I}}(B) = 0$ for all $w>0,$ we have $H^{l}_{\mathcal{I}(u,v)}(A)=0$ and $G_{\mathcal{I}(u,v)}(A) = 0$ for all nonsingular TN matrices $A.$ This completes the proof for nonsingular TN. To complete the proof for all TN matrices we use the density of nonsingular TN matrices in TN matrices and the continuity of the determinant.
\end{proof}

Lemma~\ref{Main-lemma-1} validates the action of the index-row (and index-column) operations $\Ro_{(u,v)}$ (and $\Co_{(u,v)}).$ We shall see that these operations have significant applications in additive (Sections~\ref{Section-add-ineq} and \ref{Proofof-Section-add-ineq}) and multiplicative inequalities (Sections~\ref{Section-multi-results} and \ref{Multi-proofs}). For now, let us document an important consequence of Lemma~\ref{Main-lemma-1} that significantly impacts the proof of Theorems~\ref{Add-thm-1}, \ref{Add-thm-2}, and \ref{Add-thm-3}.

\begin{ulemma}\label{Main-lemma-3}
Let $n,p,q\geq 1$ be integers and $\mathcal{I}$ be a nonempty indexing set. Suppose $J_{i1},J_{i2}$ are subsets of $[n]$ such that $|J_{i1}|=p$ and $|J_{i2}|=q$ for all $i\in \beta.$ Define the following for all $P,Q\subseteq [n]$ provided $|P|=p$ and $|Q|=q$:
\begin{align}\label{Main-lemma-3-ineq}
F_{(P,Q)}(A):=\sum_{i\in \mathcal{I}} \lambda_{i} \det A\big{(}P\big{|}J_{i1}\big{)} \det A\big{(}Q\big{|}J_{i2}\big{)}, \mbox{ for all } n\times n \mbox{ matrices }A.
\end{align}
If $F_{(P,Q)}(A)\geq 0$ for all $n\times n$ totally nonnegative matrices $A,$ for some $P,Q\subseteq [n],$ then the following assertions hold:
\begin{itemize}
\item[$(1)$] For all $n\times n$ totally nonnegative matrices $A,$
\begin{align*}
F_{\big{(}P(u,v),Q(u,v)\big{)}}(A)\geq 0, \mbox{ for all consecutive }u,v\in [n].
\end{align*}
\item[$(2)$] For all $n\times n$ totally nonnegative matrices $A,$
\begin{align*}
F_{(P,Q)}(A)\geq 0, \mbox{ for all }P,Q\subseteq [n], \mbox{ where, either }P\subseteq Q \mbox{ or }Q\subseteq P.
\end{align*}
\end{itemize}
\end{ulemma}

The complete applicability of Lemma~\ref{Main-lemma-1} and index-row/column operations depends on specific inequalities that we work with. Lemma~\ref{Main-lemma-3}, however, is an attempt to demonstrate the applicability of these on inequalities with sums of products of two minors with certain specificity in the row indices. We shall see similar situations when we prove Theorem~\ref{Add-thm-1} and Theorem~\ref{Add-thm-3} in Section~\ref{Proofof-Section-add-ineq}. For the next and several other proofs in the paper, we frequently use the order preserving maps, which we define below:

\begin{defn}[Order preserving maps]
Suppose $n\geq 1$ is an integer, and $M,N\subseteq [n]$ are nonempty with $|M|=|N|.$ Then a bijective map $\alpha:M\to N$ is called \textit{order preserving} provided $\alpha(m_1) \leq \alpha(m_2)$ for all $m_1 \leq m_2 \in M.$ 

It is easy to see that, for given $M,N,$ this map is unique.
\end{defn}

\begin{proof}[Proof of Lemma~\ref{Main-lemma-3}]
The proof of the first statement follows from Lemma~\ref{Main-lemma-1}. We now prove the second statement. Without loss of generality, suppose $p\leq q,$ and let $F_{(P,Q)}(A)\geq 0$ for all $n\times n$ TN matrix $A,$ where $p=|P|$ and $q=|Q|.$ Let $X=\{x_1 < \cdots < x_p\},$ and $Y=\{y_1 < \cdots < y_q\}$ be subsets of $[n]$ such that $X\subseteq Y.$ Suppose $X\cap Y =\{y_{k_1} < \cdots < y_{k_p}\}.$

We aim to prove that there exists a sequential action of index-row operations $\Ro_{(u,v)}$ on $\Big{(}\big{(}P,Q\big{)}\Big{)}$ results into $\Big{(}\big{(}X,Y\big{)}\Big{)}.$ Let $\Ro:=\Ro_{(2,1)}\circ \cdots \circ \Ro_{(n-1,n-2)}\circ \Ro_{(n,n-1)},$ and note that
\begin{align*}
\Ro^{q}\Big{(}\big{(}P,Q\big{)}\Big{)}=\Big{(}\big{(}[1,p],[1,q]\big{)}\Big{)}.
\end{align*}
\noindent This means $F_{\big{(}[1,p],[1,q]\big{)}}(A)\geq 0$ for all $n\times n$ TN matrices $A.$ If $X=[1,p]$ and $Y=[1,q],$ then this concludes the proof. If not, then let $\alpha:[1,p] \to \{k_1, \ldots, k_p\}$ be the unique order preserving map. For this map, define 
$$
\Ro_{\A,j}:= \Ro_{(\alpha(j)-1,\alpha(j))} \circ \cdots \circ \Ro_{(j,j+1)}, \mbox{ and } \Ro_{\A}:=\Ro_{\A,1}\circ \cdots \circ \Ro_{\A,p},
$$
for all $j\in [1,p].$ Finally, note that there exists an order preserving map $\beta: [1,q] \to Y.$ For this map, define 
$$
\Ro_{\beta,j}:= \Ro_{(\beta(j)-1,\beta(j))} \circ \cdots \circ \Ro_{(j,j+1)}, \mbox{ and } \Ro_{\beta}:=\Ro_{\beta,1}\circ \cdots \circ \Ro_{\beta,q},
$$
for all $j\in [1,q].$ Hence
$$
\Ro_{\beta} \circ \Ro_{\A} \Big{(}\big{(}[1,p],[1,q]\big{)}\Big{)} = \Big{(}\big{(}X,Y\big{)}\Big{)}.
$$
Therefore $F_{(X,Y)}(A)\geq 0$ for all $n\times n$ TN matrices $A.$  This completes the proof.
\end{proof}

\begin{remark}[Applying $\Ro_{(u,v)}$ and $\Co_{(u,v)}$]\label{Main-lemma-3-remark}
We saw in the proof of Lemma~\ref{Main-lemma-3} that compositions of index-row operations $\Ro_{(u,v)}$ were applied on the given inequality to deduce other inequalities. We can apply index-column operations $\Co_{(u,v)}$ as well in Lemma~\ref{Main-lemma-3}, and depending on $J_{i1}$ and $J_{i2},$ we may obtain novel and nontrivial inequalities over TN matrices. We can see this demonstrated very well in the proof of Theorem~\ref{Add-thm-3} in Section~\ref{Proofof-Section-add-ineq} (see also Remark~\ref{Column-op-rem}). However, this method of obtaining new inequalities from the given ones will be a common occurrence in almost all the proofs in the coming sections.
\end{remark}

\begin{remark}[Hierarchy among inequalities]\label{Hierarchy-rem-1}
Let us now recall Remark~\ref{Hierarchy-rem-0} in which we alluded to the fact that the index-row/column operations introduce a certain hierarchy in the set of all inequalities over totally nonnegative matrices. More generally, the inequality $F_{\mathcal{I}}(A)\geq 0$ is higher than $G_{\mathcal{I}(u,v)}(A)\geq 0$ in Lemma~\ref{Main-lemma-1} as the later can be derived from the former via index-row operation(s). Since the reverse of this, however, may or may not be possible, it is natural to ask for the inequalities at the top of this hierarchy -- which are those that cannot be derived from any other inequality via these operations. Observe that, because of the way these index-row/column operations function, certainly the inequalities which are \textit{complementary} lie at the top of this hierarchy. 
\end{remark}

\begin{defn}[Complementary determinantal inequalities]\label{comp-ineq}
Suppose $n\geq 1$ is an integer, and let $\mathcal{I}$ be a nonempty indexing set, and $\gamma:\mathcal{I} \to \Z_{>0}.$ Suppose $I_{ij},J_{ij}\subseteq [n],$ such that $|I_{ij}|=|J_{ij}|,$ for $i\in \mathcal{I}$ and $j\in [\gamma(i)].$ Then for real $\lambda_{i},~i \in \mathcal{I},$ suppose the following inequality holds:
\begin{align*}
F_{\mathcal{I}}(A):=\sum_{i\in \mathcal{I}} \lambda_{i} \prod_{j=1}^{\gamma(i)} \det A\big{(}I_{ij}\big{|}J_{ij}\big{)}\geq 0,
\end{align*}
for all $n\times n$ totally nonnegative matrices $A.$ This inequality is called \textit{complementary} provided,
\begin{align*}
I_{ij_{1}}\cap I_{ij_{2}}=J_{ij_{1}}\cap J_{ij_{2}}=\emptyset, \mbox{ for all distinct } j_1,j_2\in [\gamma(i)], \mbox{ for every } i\in \mathcal{I}.
\end{align*}
\end{defn}

Definition~\ref{comp-ineq} implies that the classical Gantmacher--Krein inequalities~\eqref{GK-ineq}, the inequalities in Theorem~\ref{Add-thm-1}(1), and the inequalities in Theorem~\ref{Add-thm-3} for $P=Q^{\mathsf{c}}=[1,d]$ are all complementary. In the next section we will see that these complementary inequalities are playing the key role in deriving the other inequalities in each result via the index-row operations, however the reverse, as mentioned earlier as well, is not possible.

\begin{remark}[Relations]\label{alg-rel}
It is natural to seek basic algebraic relations governing the action of these index-row/index-column operations. Mirroring the seminal relations due to Berenstein--Fomin--Zelevisky \cite{BFZ} we have the following relations, which are easy to verify. Suppose $n\geq 1$ be an integer, then:
\begin{itemize}
\item[$(a)$] $\Ro_{(u_1,v_1)}\circ \Ro_{(u_2,v_2)} = \Ro_{(u_2,v_2)}\circ \Ro_{(u_1,v_1)}$ for all consecutive $u_1,v_1\in [n]$ and consecutive $u_2,v_2\in [n]$ whenever $\{u_1,v_1\}\cap\{u_2,v_2\}=\emptyset.$
\item[$(b)$] $\Ro_{(u,v)}\circ \Ro_{(u,v)} = \Ro_{(u,v)}$ for all consecutive $u,v\in [n].$
\item[$(c)$] $\Ro_{(u_1,v_1)}\circ \Co_{(u_2,v_2)} = \Co_{(u_2,v_2)}\circ \Ro_{(u_1,v_1)}$ for all consecutive $u_1,v_1\in [n]$ and consecutive $u_2,v_2\in [n].$
\end{itemize} 
Similar relations are true when we swap the roles of $\Ro$ and $\Co.$ The important Braid relations, on the other hand, do not seem to hold and may need some essential preconditioning. We leave investigating this case and finding other algebraic connections to our future work. 
\end{remark}

\section{Additive inequalities: proofs}\label{Proofof-Section-add-ineq}

In this section we prove the additive inequalities discussed in Section~\ref{Section-add-ineq}. For the following proof we use the well known Sylvester's identity (see \cite{FJb}), in addition to the index-row operations $\Ro_{(u,v)}.$


\begin{proof}[Proof of Theorem~\ref{Add-thm-1}]
We provide two proofs of inequality \eqref{Add-thm-1-ineq-2}. We may apply Theorem~\ref{Add-thm-2} with $S=\{i\}, T=\emptyset$ and $p=j$ for all $i\neq j.$ The second proof utilizes the index-row operation $\Ro_{(u,v)}.$ Consider the Gantmacher--Krein inequality~\eqref{GK-ineq}, and apply the index-row operation $\Ro_{(1,2)}$ on it. This yields,
\begin{align*}
(-1)^{1+l}\sum_{i=1}^{l}(-1)^{1+k}a_{2k}\det A_{1k} \geq 0, \mbox{ for all }l\in [n], \mbox{ for all } n\times n \mbox{ TN matrices }A.
\end{align*}
Using Lemma~\ref{Main-lemma-3} on above, we have 
\begin{align*}
(-1)^{1+l}\sum_{i=1}^{l}(-1)^{1+k}a_{ik}\det A_{jk} \geq 0, \mbox{ for all }l\in [n], \mbox{ for all } n\times n \mbox{ TN matrices }A,
\end{align*}
for all $i,j\in [n]$ with $i\neq j.$

For the proof of \eqref{Add-thm-1-ineq-1} we adapt the ideas in \cite{GK2}. Note that 
$$
(-1)^{1+i}\sum_{k=1}^{n}a_{ik}^* =0, \mbox{ for all } i\in [n].
$$
In addition to this, recall that if $A=(a_{ij})$ is $n\times n$ TN, then $A':=(a_{ij}'),$ where $a_{ij}'=a_{n+1-i,n+1-j}$ is also TN. These observations imply that it is enough to prove inequalities~\eqref{Add-thm-1-ineq-1} for $l\leq i.$ Equivalently, showing
\begin{align}\label{ProofofA-E1}
(-1)^{i+l}\sum_{k=1}^{l}(-1)^{i+k}a_{ik}\det A_{ik}\geq 0
%
\end{align}
for all $l\leq i-1,$ and all $i\in [2,n].$ Now for $i\in [2,n],$ and $l\in [1,i-1],$ define an $n\times n$ matrix $B_{il}:=(b_{pq}),$ where the entries 
$$
b_{pq}=
\begin{cases}
0 & \mbox{if }p=i \mbox{ and } q\in [l+1,n], \mbox{ and} \\
a_{pq} & \mbox{otherwise}.
\end{cases}
$$
We can now see that showing \eqref{ProofofA-E1} is equivalent to proving
\begin{equation}\label{ProofofA-InducHypo}
(-1)^{i+l}\det B_{il}\geq 0 \mbox{ for all } i\in [2,n], \mbox{ and } l\in [1,i-1].
\end{equation}
Note that equation~\eqref{ProofofA-InducHypo} holds for all TN matrices of size up to $3\times 3.$ So we assume that equation~\eqref{ProofofA-InducHypo} holds for all TN matrices of size upto $n-1 \times n-1.$

Now if $\det A_{ik}=0$ for all $k\in [l],$ then there is nothing to prove. So suppose $\det A_{ik_1} > 0$ for some $1\leq k_1 \leq l.$ This means there exist $i_0,k_0\in [n]$ such that $k_0 > l$ and $\det A \big{(} \{i,i_0\}^{\mathsf{c}} \big{|} \{k_1,k_0\}^{\mathsf{c}} \big{)}>0.$ Construct matrix $C$ by swapping row $i_0$ \& row $k_1$ and column $k_0$ \& column $i$ of $B_{il}.$ Then, using Sylvester's identity (see \cite{FJb}) we have, 
\begin{align*}
\det C  \det C \big{(}\{i,k_1\}^{\mathsf{c}} \big{|}\{i,k_1\}^{\mathsf{c}}\big{)}=& \det C \big{(} \{i\}^{\mathsf{c}} \big{|}\{i\}^{\mathsf{c}}\big{)} \det C \big{(} \{k_1\}^{\mathsf{c}} \big{|}\{k_1\}^{\mathsf{c}}\big{)} \\
&\hspace*{2.5cm}
- \det C \big{(} \{i\}^{\mathsf{c}} \big{|}\{k_1\}^{\mathsf{c}}\big{)}  \det C \big{(} \{k_1\}^{\mathsf{c}} \big{|}\{i\}^{\mathsf{c}}\big{)}.
\end{align*}
This becomes,
\begin{align*}
\det B_{il} \det B_{il} \big{(}\{i,i_0\}^{\mathsf{c}} \big{|}\{k_1,k_0\}^{\mathsf{c}}\big{)} &= \det B_{il} \big{(} \{i\}^{\mathsf{c}} \big{|}\{k_0\}^{\mathsf{c}}\big{)} \det B_{il} \big{(} \{i_0\}^{\mathsf{c}} \big{|}\{k_1\}^{\mathsf{c}}\big{)} \\
&\hspace*{2.5cm} - \det B_{il} \big{(} \{i\}^{\mathsf{c}} \big{|}\{k_1\}^{\mathsf{c}}\big{)}  \det B_{il} \big{(} \{i_0\}^{\mathsf{c}} \big{|}\{k_0\}^{\mathsf{c}}\big{)}. 
\end{align*}
This can further be written as,
\begin{align*}
(-1)^{i+l}\det B_{il} & \det A \big{(}\{i,i_0\}^{\mathsf{c}} \big{|}\{k_1,k_0\}^{\mathsf{c}}\big{)} =A \big{(} \{i\}^{\mathsf{c}} \big{|}\{k_0\}^{\mathsf{c}}\big{)} \Big{(} (-1)^{i+l-2} \det B_{il} \big{(} \{i_0\}^{\mathsf{c}} \big{|}\{k_1\}^{\mathsf{c}}\big{)} \Big{)} \\
&\hspace*{4.5cm} + A \big{(} \{i\}^{\mathsf{c}} \big{|}\{k_1\}^{\mathsf{c}}\big{)}\Big{(} (-1)^{i+l-1} \det B_{il} \big{(} \{i_0\}^{\mathsf{c}} \big{|}\{k_0\}^{\mathsf{c}}\big{)} \Big{)}.
\end{align*}
Using the induction hypothesis and that $\det A \big{(}\{i,i_0\}^{\mathsf{c}} \big{|}\{k_1,k_0\}^{\mathsf{c}}\big{)} >0,$ we have $(-1)^{i+l}B_{il} \geq 0$, which completes the proof.
\end{proof}

For the following proof, we use the bidiagonal factorization Theorem~\ref{TN-classification} and  Notion~\ref{main-idea}.

\begin{proof}[Proof of Theorem~\ref{Add-thm-2}]
Fix $m\in [n],$ and define $F_{(T,S,p)}^{l}(A)$ for all $n\times n$ TN matrices $A,$ for all $l\in [m],$ and all $S,T\subseteq [n]$ and $p\in [n],$ such that $S\subseteq [n]\setminus \{p\}$ and $|S|=|T|+1:$
\begin{align}\label{ProofofTheoremB-E1}
F_{(T,S,p)}^{l}(A):=\sum_{k=1}^{l} (-1)^{1+k} \det A\big{(}S\big{|}T\cup \{v_k\}\big{)} \det A\big{(}[n]\setminus\{p\}\big{|}[n]\setminus\{v_k\}\big{)},
\end{align}
where $V:=\{v_1<\cdots<v_m \}:=[n]\setminus T.$

Note that $F_{(T,S,p)}^{l}(D)=0$ for all $n\times n$ diagonal TN matrices, for all $T,S,p,$ and $l$ satisfying conditions in \eqref{ProofofTheoremB-E1}. Let $A$ be an $n\times n$ nonsingular TN matrix which is not diagonal. Then there exists $w>0$ and consecutive integers $u,v\in [n]$ such that $A=B(\I+w\El_{uv}),$ where $B$ is nonsingular TN. Recall Notion~\ref{main-idea}, and consider the following computations:\\
\noindent\textbf{Case 1.} $u \in T:$ It would not be difficult to see that
\begin{align*}
&\sum_{k=1}^{l} (-1)^{1+k} \det A\big{(}S\big{|}T\cup \{v_k\}\big{)} \det A\big{(}[n]\setminus\{p\}\big{|}[n]\setminus\{v_k\}\big{)} \\
&\hspace*{4cm}=~\sum_{k=1}^{l} (-1)^{1+k} \det B\big{(}S\big{|}T\cup \{v_k\}\big{)} \det B\big{(}[n]\setminus\{p\}\big{|}[n]\setminus\{v_k\}\big{)}.
\end{align*}
This implies, $$F_{(T,S,p)}^{l}(A)= F_{(T,S,p)}^{l}(B).$$

\noindent\textbf{Case 2.} $u\in V$ and $v\in T:$ 

Suppose $u=v_i,$ and write
\begin{align*}
&\sum_{k=1}^{l} (-1)^{1+k} \det A\big{(}S\big{|}T\cup \{v_k\}\big{)} \det A\big{(}[n]\setminus \{p\} \big{|} [n]\setminus\{v_k\}\big{)} \\
&\hspace*{1cm}=\sum_{k=1,k \neq i}^{l} (-1)^{1+k} \det A \big{(}S\big{|}T \cup \{v_k\} \big{)} \det A \big{(}[n]\setminus \{p\} \big{|} [n]\setminus\{v_k\}\big{)} \\
&\hspace*{3cm} + (-1)^{1+i} \det A \big{(}S\big{|}T\cup \{u\}\big{)} \det A\big{(}[n]\setminus\{p\}\big{|}[n]\setminus\{u\}\big{)} \\
&\hspace*{1cm}=\sum_{k=1,k \neq i}^{l} (-1)^{1+k} \bigg{(} \det B\big{(}S\big{|}T \cup \{v_k\}\big{)}\\
&\hspace*{3cm} + w \det B \big{(}S\big{|}\big{(}\big{(}T\setminus\{v\}\big{)}\cup\{u\}\big{)}\cup \{v_k\}\big{)} \bigg{)} \det B\big{(}[n]\setminus\{p\}\big{|}[n]\setminus\{v_k\}\big{)} \\
&\hspace*{6cm} + (-1)^{1+i} \det B \big{(} S\big{|}T\cup \{u\}\big{)} \bigg{(} \det B\big{(}[n]\setminus\{p\}\big{|}[n]\setminus\{u\}\big{)}\\
&\hspace*{9cm}+ w\det B\big{(}[n]\setminus \{p\}\big{|}[n]\setminus\{v\}\big{)} \bigg{)}\\
&\hspace*{1cm}=\sum_{k=1}^{l} (-1)^{1+k} \det B\big{(}S\big{|}T\cup \{v_k\}\big{)} \det B\big{(}[n]\setminus \{p\} \big{|}[n]\setminus\{v_k\}\big{)}\\
&\hspace*{3cm}+w \sum_{k=1}^{l} (-1)^{1+k} \det B\big{(}S\big{|}T(q,p)\cup \{v_k'\}\big{)} \det B\big{(}[n]\setminus\{p\} \big{|} [n]\setminus\{v_k'\}\big{)},
\end{align*}
where $T(q,p)=\big{(}T\setminus \{q\}\big{)}\cup\{p\},$ and $v_{k}'=\begin{cases} v_{k} & \mbox{if }k\neq i, \mbox{ and} \\ q & \mbox{if }k=i.\end{cases}$\\ 
Note that $V(p,q):=\big{(}V\setminus\{p\}\big{)}\cup \{q\}=\{v_{k}':k\in [m]\}=[n]\setminus T(q,p).$ Hence, we have
$$
F_{(T,S,p)}^{l}(A)= F_{(T,S,p)}^{l}(B) + w F_{(T(q,p),S,p)}^{l}(B).
$$

\noindent\textbf{Case 3.} $u,v\in V:$ 

Suppose $v_{i}=u$ and $v_{i+1}=v.$ In this case, we have
\begin{align*}
&\det A\big{(}S\big{|}T\cup \{v_k\}\big{)} \det A\big{(}[n]\setminus \{p\}\big{|}[n]\setminus\{v_k\}\big{)} \\
&\hspace*{3cm}= 
\begin{cases}
\det B\big{(}S\big{|}T\cup \{v_k\}\big{)} \det B\big{(}[n]\setminus \{p\}\big{|}[n]\setminus\{v_k\}\big{)} & \mbox{if } k \not\in \{i,i+1\},\\
&\\
\det B\big{(}S\big{|}T\cup \{v_i\}\big{)} \Big{(} \det B\big{(}[n]\setminus\{p\}\big{|}[n]\setminus\{v_i\}\big{)} \\ \hspace*{2cm}+~ w \det B\big{(}[n]\setminus\{p\}\big{|}[n]\setminus\{v_{i+1}\}\big{)}\Big{)} & \mbox{if } k = i,\\
&\\
\det B\big{(}[n]\setminus\{p\}\big{|}[n]\setminus\{v_{i+1}\}\big{)} \Big{(} \det B\big{(}S\big{|}T\cup \{v_{i+1}\}\big{)} \\
\hspace*{2cm}+~w\det B\big{(}S\big{|}T\cup \{v_i\}\big{)}\Big{)} & \mbox{if } k = i+1.\\
\end{cases}
\end{align*} 
This implies,
\begin{align*}
&F_{(T,S,p)}^{l}(A)\\
&\hspace*{0.5cm}
=\begin{cases}
F_{(T,S,p)}^{l}(B) ~+~(-1)^{1+l}w\det B \big{(}S\big{|}T\cup \{v_l\}\big{)} \det B \big{(}[n]\setminus \{p\}\big{|}[n]\setminus\{v_{l+1}\}\big{)} & \mbox{ if }l=i,\\
& \\
F_{(T,S,p)}^{l}(B) & \mbox{ if }l\neq i.
\end{cases}
\end{align*}
Similar is the result when $v_{i-1}=v$ and $v_i=u.$ 

Recall that Theorem~\ref{TN-classification} states that for every $n\times n$ non-diagonal nonsingular TN matrix $A,$ there exist pairs of consecutive integers $(a_1,b_1),\ldots,(a_N,b_N),$ nonnegative real numbers $w_1,\ldots,w_k,$ and a diagonal matrix $D$ with positive diagonals such that 
\begin{align}\label{ProofofTheoremB-E2}
A=D\prod_{k=1}^{N}(\I+w_k\El_{a_k,b_k}).
\end{align}
Using \eqref{ProofofTheoremB-E2}, computations from the three cases above, and the fact that $F_{(T,S,p)}^{l}(D)=0$ for all $T,S,p$ and $l$ satisfying conditions in \eqref{ProofofTheoremB-E1}, we can say that there exist subsets $T_k\subseteq [n]$ with $|T_k|=|S|-1,$ for some integer $k\geq 1,$ and nonnegative real numbers $y_{k_1},z_{k_2},$ for integers $k_1,k_2,N_1,N_2 \geq 1$ such that
\begin{align*}
&(-1)^{1+l}F_{(T,S,p)}^{l}(A)\\
&\hspace*{1cm}= (-1)^{1+l} \sum_{k_1=1}^{N_1}y_{k_1} F_{(T_{k_1},S,p)}^{l}(D) 
\\ &\hspace*{4cm} +\sum_{k_2=1}^{N_2} z_{k_2} \det D \big{(}S\big{|}T_{k_2}\cup\{v_{{k_2},l}\}\big{)} \det D \big{(}[n]\setminus\{p\}\big{|}[n]\setminus\{v_{{k_2},l+1}\}\big{)}\\ 
& \hspace*{1cm}= \sum_{k_2=1}^{N_2} z_{k_2} \det D \big{(}S\big{|}T_{k_2}\cup\{v_{{k_2},l}\}\big{)} \det D \big{(}[n]\setminus\{p\}\big{|}[n]\setminus\{v_{{k_2},l+1}\}\big{)}\geq 0,
\end{align*}
for all non-diagonal nonsingular $n\times n$ TN matrices $A.$ This completes the proof for all $n\times n$ nonsingular TN matrices. To complete the proof for all TN matrices, we invoke the density of nonsingular TN matrices and the continuity of determinant.  
\end{proof}

The proof of Theorem~\ref{Add-thm-3} is completed in two steps. Step one would be to prove the inequality~\eqref{Add-thm-3-ineq} for the complementary case $P=Q^{\mathsf{c}}=[1,d]$ in Lemma~\ref{Add-thm-3-lemma}. This is completed using the identification of totally nonnegative matrices with planar networks -- see the construction in Section~\ref{Section-TN}. This gives us the higher/complementary inequality that we require to use in step two, in which we use the action of index-row operations $\Ro_{(u,v)}$ on \eqref{Add-thm-3-lemma-ineq} (in Lemma~\ref{Add-thm-3-lemma}) to conclude the proof of Theorem~\ref{Add-thm-3} for all other values of $P$ and $Q.$

\begin{ulemma}\label{Add-thm-3-lemma}
Let $n\geq 2,$ $d\in [1,n-1],$ and $l\in [1,d]$ be integers. Suppose $I_{d} := [1,d],$  and $J_{dl} := [n-d,n]\setminus\{n-d+l\}$ for $l\in [0,d].$ Then
\begin{align}\label{Add-thm-3-lemma-ineq}
(-1)^{1+l}\sum_{k=0}^{l}(-1)^{1+k}\det A\big{(}I_{d}\big{|}J_{dk}\big{)} \det A\big{(}[n]\setminus I_{d}\big{|}[n]\setminus J_{dk}\big{)}\geq 0,
\end{align}
for all $l\in [0,d],$ and all $n\times n$ totally nonnegative matrices $A.$
\end{ulemma}
\begin{proof} Since nonsingular TN matrices are dense in TN matrices and the determinant is continuous, it is enough to prove these inequalities for $A$ a nonsingular TN matrices. Then $A$ factors according to Theorem~\ref{TN-classification}:
$
A=LDR,
$
where $L$ and $R$ respectively are the products of factors on the left and right of $D$ in equation~\eqref{TN-classification-eqn}. Each factor $L,$ $D,$ and $R$ is nonsingular TN, and we respectively use $N_L,$ $N_D,$ and $N_R$ to refer to the planar networks. The construction in Section~\ref{Section-TN} implies that the planar network for $A,$ say $N_A,$ is the concatenation (from left to right) of $N_L,$ $N_D,$ and $N_R,$ where $N_D$ is the network with the blue edges, $N_L$ is the network on the left side of the blue edges, and $N_R$ is the network on the right side of the blue edges -- all in Figure~\ref{Planar-bw}. So, for this proof, we refer to Figure~\ref{Planar-bw} henceforth.

The blue edges are labelled and weighted as $d_1,$ $d_2,$ $\ldots,$ $d_n$ in $N_D.$ The network $N_A$ has horizontal and non-horizontal edges. These contain $n$ directed horizontal paths from source $k$ to sink $k$ (that goes via $d_k)$ for $k\in [n].$ We refer to the horizontal path from source $k$ to sink $k$ as level-$k,$ for all $k\in [n].$ The network $N_L$ has $n-1$ non-horizontal directed paths, labelled as $l_1,$ $l_2,$ $\ldots,$ $l_{n-1}.$ The edge in $l_k$ that leads to a vertex on level-$j$ is labelled and weighted as $l_{k}^{j}.$ Similarly, network $N_R$ has $n-1$ non-horizontal directed paths, labelled as $u_1,$ $u_2,$ $\ldots,$ $u_{n-1}.$ The edge in $u_k$ that leads to a vertex on level-$j$ is labelled and weighted as $u_{k}^{j}.$ Now define
\begin{itemize}
\item[$(1)$] $Q_{d1}:=Q_{d1}(A)$ to be the total weighted sum over the family of disjoint paths from source vertices $[1,d]$ to sink vertices $[n-d,n-1]$ in $N_A.$
\item[$(2)$] $R_{[n-d+l,n-1]^{\uparrow 1}}^{u_{m_1}^{n-d+l+1}}:=R_{[n-d+l,n-1]^{\uparrow 1}}^{u_{m_1}^{n-k+l+1}}(A)$ to be the total weighted sum over the family of disjoint paths from source vertices $[n-d+l,n-1]$ to sink vertices $[n-d+l+1,n]$ in $N_R$ that use $u_{m_1}^{n-k+l+1},$ where $m_1\in [l+1].$  
\end{itemize}
With these notations, using Lindstr\"{o}m Lemma~\ref{Lin-lemma}, we have
\begin{align*}
\det A\big{(}I_{d}\big{|}J_{dl}\big{)}=Q_{d1}\sum_{m_1=1}^{l+1}R_{[n-d+l,n-1]^{\uparrow 1}}^{u_{m_1}^{n-d+l+1}}, \mbox{ for all } l\in [0,d], \mbox{ where }\sum_{m_1=1}^{d+1}R^{u_{m_1}^{n+1}}_{[n,n-1]^{\uparrow 1}}:=1.
\end{align*}
For the other set of notations, define the following:
\begin{itemize}
\item[$(1)$] $Q_{d2}:=Q_{d2}(A)$ to be the total weighted sum over the family of disjoint paths from source vertices $[d+1,n-1]$ to sink vertices $[1,n-d-1]$ in $N_A.$
\item[$(2)$] $P_{n\to n-d+l}^{d_{n-d+m_2}}:=P_{n\to n-d+l}^{d_{n-d+m_2}}(A)$ to be the total weighted sum of paths from source $n$ to sink $n-d+l$ via diagonal edge $d_{n-d+m_2},$ for $m_2 \in [0,l].$
\end{itemize}
With these notations, using Lindstr\"{o}m Lemma~\ref{Lin-lemma}, we have
\begin{align*}
\det A \big{(}[n]\setminus I_{d}\big{|}[n]\setminus J_{dl}\big{)}= Q_{d2}\sum_{m_2=0}^l P_{n\to n-d+l}^{d_{n-d+m_2}},\mbox{ for } d\in [1,n-1],
\end{align*}
$ \mbox{ where } Q_{(n-1),2}:= 1 \mbox{ for } l\in [0,d].$

Suppose $Q:=Q_{d1}Q_{d2},$ and let $\mathcal{R}^{u_k^{n-d+l}}_{n-d+m_2}$ be the set of all paths $p,$ with weights $p,$ in $N_R$ which starts from level-$(n-d+m_{2})$ and goes to level-$(n-d+l)$ using $u_{k}^{n-d+l},$ where $m_2\in [0,l],$ in Figure~\ref{Planar-bw}. We aim to prove that for all $0\leq l\leq d,$
\begin{align*} 
&\sum_{k=0}^{l} (-1)^{1+k}\det A\big{(}I_{d}\big{|}J_{dk}\big{)}\det A \big{(}[n]\setminus I_{d}\big{|}[n]\setminus J_{dk}\big{)}\\
&\hspace*{2cm}=(-1)^{1+l}Q \sum_{m_1=1}^{l+1}R^{u_{m_1}^{n-d+l+1}}_{[n-d+l,n-1]^{\uparrow 1}} \sum_{m_2=0}^{l}P^{d_{n-d+m_2}}_{n \to n-d+m_2}\sum_{i=m_1}^{n-d+(l-1)}\sum_{p\in \mathcal{R}^{u_i^{n-d+l}}_{n-d+m_2}} p,
\end{align*}
where $\sum_{i=m_1}^{n-d+l-1}\sum_{_{p\in \mathcal{R}^{u_i^{n-d+l}}_{n-d+l}}} p:=1.$ We use induction for the proof. Let us verify for $l=1$:
\begin{align*}
&\sum_{k=0}^{1} (-1)^{1+k}\det A \big{(}I_{d}\big{|}J_{dk}\big{)}\det A\big{(}[n]\setminus I_{d} \big{|}[n]\setminus J_{dk}\big{)} \\ 
&\hspace*{1cm}= Q \Big{(} - R^{u_1^{n-d+1}}_{[n-d,n-1]^{\uparrow 1}}  P^{d_{n-d}}_{n\to n-d} + \sum_{m_1=1}^{2}R^{u_{m_1}^{n-d+2}}_{[n-d+1,n-1]^{\uparrow 1}} \sum_{m_2=0}^1 P^{d_{n-d+m_2}}_{n\to n-d+1}\Big{)}\\
&\hspace*{1cm}=Q \Big{(} - u^{n-d+1}_{1} R^{u_2^{n-d+2}}_{[n-d+1,n-1]^{\uparrow 1}}  P^{d_{n-d}}_{n\to n-d} \\
&\hspace*{5cm} + \sum_{m_1=1}^{2}R^{u_{m_1}^{n-d+2}}_{[n-d+1,n-1]^{\uparrow 1}} \big{(} P^{d_{n-d}}_{n\to n-d+1} + P^{d_{n-d+1}}_{n\to n-d+1} \big{)} \Big{)} \\
&\hspace*{1cm}=-Q\Big{(} u^{n-d+1}_{1} R^{u_2^{n-d+2}}_{[n-d+1,n-1]^{\uparrow 1}}  P^{d_{n-d}}_{n\to n-d} \\
&\hspace*{5cm} - \sum_{m_1=1}^{2}R^{u_{m_1}^{n-d+2}}_{[n-d+1,n-1]^{\uparrow 1}} \big{(} P^{d_{n-d}}_{n \to n-d}\sum_{i=1}^{n-d} u_{i}^{n-d+1} + P^{d_{n-d+1}}_{n\to n-d+1} \big{)} \Big{)}\\
&\hspace*{1cm}= Q \Big{(} R^{u_{1}^{n-d+2}}_{[n-d+1,n-1]^{\uparrow 1}} \big{(} P^{d_{n-d}}_{n \to n-d}\sum_{i=1}^{n-d} u_{i}^{n-d+1} + P^{d_{n-d+1}}_{n\to n-d+1} \big{)}\\
&\hspace*{5cm} + R^{u_{2}^{n-d+2}}_{[n-d+1,n-1]^{\uparrow 1}} \big{(} P^{d_{n-d}}_{n \to n-d}\sum_{i=2}^{n-d} u_{i}^{n-d+1} + P^{d_{n-d+1}}_{n\to n-d+1} \big{)} \Big{).}
\end{align*}
Now suppose the claim is true for $l-1 \in [1,d-1],$ and proceed to verify it for $l$:
\begin{align*}
&\sum_{k=0}^{l} (-1)^{1+k} \det A \big{(}I_{d}\big{|}J_{dk}\big{)}\det A\big{(}[n]\setminus I_{d}\big{|}[n]\setminus J_{dk}\big{)}\\
&= \sum_{k=0}^{l-1} (-1)^{1+k} \det A \big{(}I_{d}\big{|}J_{dk}\big{)}\det A\big{(}[n]\setminus I_{d}\big{|}[n]\setminus J_{dk}\big{)} \\
&\hspace*{5cm} + (-1)^{1+l} \det A \big{(}I_{d}\big{|}J_{dl}\big{)}\det A\big{(}[n]\setminus I_{d}\big{|}[n]\setminus J_{dl}\big{)}\\
&= (-1)^{l} Q \sum_{m_1=1}^{l}R^{u_{m_1}^{n-d+l}}_{[n-d+l-1,n-1]^{\uparrow 1}} \sum_{m_2=0}^{l-1}P^{d_{n-d+m_2}}_{n \to n-d+m_2}\sum_{i=m_1}^{n-d+(l-2)}\sum_{p\in \mathcal{R}^{u_i^{n-d+l-1}}_{n-d+m_2}} p\\
& \hspace*{5cm} + (-1)^{1+l} Q \sum_{m_1=1}^{l+1}R^{u_{m_1}^{n-d+l+1}}_{[n-d+l,n-1]^{\uparrow 1}} \sum_{m_2=0}^l P^{d_{n-d+m_2}}_{n\to n-d+l} \\
&= (-1)^{1+l} Q \Bigg{(} - \sum_{m_1=2}^{l+1} u_{m_1-1}^{n-d+l} \sum_{j=m_1}^{l+1} R^{u_{j}^{n-d+l+1}}_{[n-d+l,n-1]^{\uparrow 1}} \sum_{m_2=0}^{l-1}P^{d_{n-d+m_2}}_{n \to n-d+m_2} \sum_{i=m_1-1}^{n-d+(l-1)} \sum_{p\in \mathcal{R}^{u_i^{n-d+l-1}}_{n-d+m_2}} p \\
& \hspace*{5cm}~+~ \sum_{m_1=1}^{l+1}R^{u_{m_1}^{n-d+l+1}}_{[n-d+l,n-1]^{\uparrow 1}} \sum_{m_2=0}^l P^{d_{n-d+m_2}}_{n\to n-d+l} \Bigg{)} \\
&= (-1)^{1+l} Q \Bigg{(} - \sum_{m_1=2}^{l+1} u_{m_1-1}^{n-d+l} \sum_{j=m_1}^{l+1} R^{u_{j}^{n-d+l+1}}_{[n-d+l,n-1]^{\uparrow 1}} \sum_{m_2=0}^{l-1}P^{d_{n-d+m_2}}_{n \to n-d+m_2} \sum_{i=m_1-1}^{n-d+(l-1)} \sum_{p\in \mathcal{R}^{u_i^{n-d+l-1}}_{n-d+m_2}} p\\
&\hspace*{5cm} ~+~ \sum_{m_1=1}^{l+1}R^{u_{m_1}^{n-d+l+1}}_{[n-d+l,n-1]^{\uparrow 1}} \sum_{m_2=0}^l P^{d_{n-d+m_2}}_{n\to n-d+m_2} \sum_{i=1}^{n-d+l-1} \sum_{p\in \mathcal{R}_{n-d+m_2}^{u_{i}^{n-d+l}}} p \Bigg{)} \\
&= (-1)^{1+l} Q \Bigg{(} - \sum_{m_1=2}^{l+1} \sum_{j=m_1}^{l+1} R^{u_{j}^{n-d+l+1}}_{[n-d+l,n-1]^{\uparrow 1}} \sum_{m_2=0}^{l-1}P^{d_{n-d+m_2}}_{n \to n-d+m_2} \sum_{i=m_1-1}^{n-d+(l-1)} \sum_{p\in \mathcal{R}^{u_i^{n-d+l-1}}_{n-d+m_2}} p u_{m_1-1}^{n-d+l} \\
& \hspace*{5cm}~+~ \sum_{m_1=1}^{l+1}R^{u_{m_1}^{n-d+l+1}}_{[n-d+l,n-1]^{\uparrow 1}} \sum_{m_2=0}^l P^{d_{n-d+m_2}}_{n\to n-d+m_2}\sum_{i=1}^{n-d+l-1} \sum_{p\in \mathcal{R}_{n-d+m_2}^{u_{i}^{n-d+l}}} p \Bigg{)} \\
&= (-1)^{1+l} Q \Bigg{(} - \sum_{m_1=2}^{l+1} \sum_{j=m_1}^{l+1} R^{u_{j}^{n-d+l+1}}_{[n-d+l,n-1]^{\uparrow 1}} \sum_{m_2=0}^{l-1}P^{d_{n-d+m_2}}_{n \to n-d+m_2} \sum_{p\in \mathcal{R}^{u_{m_1 -1}^{n-d+l}}_{n-d+m_2}} p \\
& \hspace*{5cm} + \sum_{m_1=1}^{l+1}R^{u_{m_1}^{n-d+l+1}}_{[n-d+l,n-1]^{\uparrow 1}} \sum_{m_2=0}^l P^{d_{n-d+m_2}}_{n\to n-d+m_2} \sum_{i=1}^{n-d+l-1} \sum_{p\in \mathcal{R}_{n-d+m_2}^{u_{i}^{n-d+l}}} p \Bigg{)}\\
&= (-1)^{1+l} Q \Bigg{(} - \sum_{m_1=2}^{l+1} \sum_{j=m_1}^{l+1} R^{u_{j}^{n-d+l+1}}_{[n-d+l,n-1]^{\uparrow 1}} \sum_{m_2=0}^{l-1}P^{d_{n-d+m_2}}_{n \to n-d+m_2} \sum_{p\in \mathcal{R}^{u_{m_1 -1}^{n-d+l}}_{n-d+m_2}} p \\
& \hspace*{2cm} +\sum_{m_1=2}^{l+1}R^{u_{m_1}^{n-d+l+1}}_{[n-d+l,n-1]^{\uparrow 1}} \Bigg{(} \sum_{m_2=0}^{l-1} P^{d_{n-d+m_2}}_{n\to n-d+m_2} \sum_{i=1}^{n-d+l-1} \sum_{p\in \mathcal{R}_{n-d+m_2}^{u_{i}^{n-d+l}}} p + P_{n\to n-d+l}^{d_{n-d+l}}\Bigg{)} \\
& \hspace*{6cm}+ R^{u_{1}^{n-d+l+1}}_{[n-d+l,n-1]^{\uparrow 1}} \sum_{m_2=0}^{l} P^{d_{n-d+m_2}}_{n\to n-d+m_2} \sum_{i=1}^{n-d+l-1} \sum_{p\in \mathcal{R}_{n-d+m_2}^{u_{i}^{n-d+l}}} p \Bigg{)}\\
&=(-1)^{1+l}Q \sum_{m_1=1}^{l+1}R^{u_{m_1}^{n-d+l+1}}_{[n-d+l,n-1]^{\uparrow 1}} \sum_{m_2=0}^{l}P^{d_{n-d+m_2}}_{n \to n-d+m_2}\sum_{i=m_1}^{n-d+(l-1)}\sum_{p\in \mathcal{R}^{u_i^{n-d+l}}_{n-d+m_2}} p,
\end{align*}
where $\sum_{i=m_1}^{n-d+l-1}\sum_{_{p\in \mathcal{R}^{u_i^{n-d+l}}_{n-d+l}}} p:=1.$

This completes this lemma and proves Theorem~\ref{Add-thm-3} for $P=Q^{\mathsf{c}}=[1,d].$
\end{proof}

Before we move on to prove Theorem~\ref{Add-thm-3} for the given other values of $P$ and $Q,$ let us demonstrate that the proof of Lemma~\ref{Add-thm-3-lemma} is capable of obtaining the classical result.

\begin{proof}[Gantmacher--Krein inequalities~\eqref{GK-ineq} via planar networks] To see how the computation in proof of Lemma~\ref{Add-thm-3-lemma} also derives the original inequalities due to  Gantmacher--Krein in Theorem~\ref{GK-thm}, let us continue it for $l=d-1=n-2.$
\begin{align*} 
&\sum_{k=0}^{n-2} (-1)^{1+k}\det A\big{(}I_{n-1}\big{|}J_{n-1,k}\big{)}\det A\big{(}[n]\setminus I_{n-1}\big{|}[n]\setminus J_{n-1,k}\big{)}\\
&\hspace*{4cm}=(-1)^{n-1} Q \sum_{m_1=1}^{n-1} R^{u_{m_1}^{n}}_{[n-1,n-1]^{\uparrow 1}} \sum_{m_2=0}^{n-2}P^{d_{1+m_2}}_{n \to 1+m_2}\sum_{i=m_1}^{n-2}\sum_{p\in \mathcal{R}^{u_i^{1+l}}_{1+m_2}} p,
\end{align*}
where $\sum_{i=m_1}^{n-2}\sum_{p\in \mathcal{R}^{u_i^{l+1}}_{m_2+1}} p = 1$ for $m_2=n-2.$

\noindent Recall the calculation for $\det A \big{(}I_{n-1}\big{|}J_{n-1,n}\big{)}\det A\big{(}[n]\setminus I_{n-1}\big{|}[n]\setminus J_{n-1,n}\big{)}:$
\begin{align*}
\det A \big{(}I_{n-1}\big{|}J_{n-1,n}\big{)}&= \det A \big{(}I_{n-1}\big{|}I_{n-1}\big{)}=Q_{n-1,1}, \mbox{ and}\\
\det A\big{(}[n]\setminus I_{n-1}\big{|}[n]\setminus J_{n-1,n}\big{)}&= \det A\big{(}[n]\setminus I_{n-1}\big{|}[n]\setminus I_{n-1}\big{)} = a_{nn}= \sum_{m_2=0}^{n-1} P^{d_{1+m_2}}_{n\to n}.
\end{align*}
Therefore,
\begin{align*}
&\det A \big{(}I_{n-1}\big{|}J_{n-1,n}\big{)}\det A\big{(}[n]\setminus I_{n-1}\big{|}[n]\setminus J_{n-1,n}\big{)} \\
&\hspace*{3cm}=Q_{n-1,1} \sum_{m_2=0}^{n-1} P^{d_{m_2+1}}_{n\to n}\\ 
&\hspace*{3cm}=Q_{n-1,1} \Bigg{(} \sum_{m_2=0}^{n-2} P^{d_{m_2+1}}_{n\to m_2+1}\sum_{m_1=1}^{n-1} \sum_{p\in \mathcal{R}^{u_{{m_1}}^{n}}_{m_2+1}} p ~+~ P^{d_{n}}_{n\to n}\Bigg{)}\\
&\hspace*{3cm}= Q_{n-1,1} \Bigg{(} \sum_{m_2=0}^{n-2} P^{d_{m_2+1}}_{n\to m_2+1}\sum_{m_1=1}^{n-1} R^{u_{m_1}^{n}}_{[n-1,n-1]^{\uparrow 1}} \sum_{i=m_1}^{n-2} \sum_{p\in \mathcal{R}^{u_{i}^{n-1}}_{m_2+1}} p ~+~ P^{d_{n}}_{n\to n}\Bigg{)}\\
&\hspace*{3cm}= Q_{n-1,1} \Bigg{(} \sum_{m_1=1}^{n-1} R^{u_{m_1}^{n}}_{[n-1,n-1]^{\uparrow 1}}  \sum_{m_2=0}^{n-2} P^{d_{m_2+1}}_{n\to m_2+1} \sum_{i=m_1}^{n-2} \sum_{p\in \mathcal{R}^{u_{i}^{n-1}}_{m_2+1}} p ~+~ P^{d_{n}}_{n\to n}\Bigg{)}\\
&\hspace*{3cm}= Q_{n-1,1} \sum_{m_1=1}^{n-1} R^{u_{m_1}^{n}}_{[n-1,n-1]^{\uparrow 1}}  \sum_{m_2=0}^{n-2} P^{d_{m_2+1}}_{n\to m_2+1} \sum_{i=m_1}^{n-2} \sum_{p\in \mathcal{R}^{u_{i}^{n-1}}_{m_2+1}} p ~+~ Q_{n-1,1} P^{d_{n}}_{n\to n}\\
&\hspace*{3cm}= Q_{n-1,1} \sum_{m_1=1}^{n-1} R^{u_{m_1}^{n}}_{[n-1,n-1]^{\uparrow 1}}  \sum_{m_2=0}^{n-2} P^{d_{m_2+1}}_{n\to m_2+1} \sum_{i=m_1}^{n-2} \sum_{p\in \mathcal{R}^{u_{i}^{n-1}}_{m_2+1}} p ~+~ \det A.
\end{align*}
This implies,
\begin{align*}
& \sum_{k=0}^{n-1} (-1)^{1+k} \det A \big{(}I_{n-1}\big{|}J_{n-1,k}\big{)}\det A\big{(}[n]\setminus I_{n-1}\big{|}[n]\setminus J_{n-1,k}\big{)} \\
&\hspace*{1cm}= \sum_{k=0}^{n-2} (-1)^{1+k} \det A \big{(}I_{n-1}\big{|}J_{n-1,k}\big{)}\det A \big{(}[n]\setminus I_{n-1}\big{|}[n]\setminus J_{n-1,d}\big{)} \\
&\hspace*{2cm} + (-1)^{n-1}\det A \big{(}I_{n-1}\big{|}J_{n-1,n-1}\big{)}\det A \big{(}[n]\setminus I_{n-1}\big{|}[n]\setminus J_{n-1,n-1}\big{)} \\
&\hspace*{1cm}=(-1)^{n-1} Q_{n-1,1} \sum_{m_1=1}^{n-1} R^{u_{m_1}^{n}}_{[n-1,n-1]^{\uparrow 1}}  \sum_{m_2=0}^{n-2} P^{d_{m_2+1}}_{n\to m_2+1} \sum_{i=m_1}^{n-2} \sum_{p\in \mathcal{R}^{u_{i}^{n-1}}_{m_2+1}} p \\
& \hspace*{2cm}+ (-1)^{n} Q_{n-1,1} \sum_{m_1=1}^{n-1} R^{u_{m_1}^{n}}_{[n-1,n-1]^{\uparrow 1}}  \sum_{m_2=0}^{n-2} P^{d_{m_2+1}}_{n\to m_2+1} \sum_{i=m_1}^{n-2} \sum_{p\in \mathcal{R}^{u_{i}^{n-1}}_{m_2+1}} p  + (-1)^{n}\det A \\
&\hspace*{1cm} = (-1)^{n} \det A. 
\end{align*}
This provides us with the original Gantmacher--Krein determinantal inequalities \eqref{GK-ineq}, and reinforces that Lemma~\ref{Add-thm-3-lemma} (and Theorem~\ref{Add-thm-3}) is indeed a generalization of this classical result.
\end{proof}
Let us now use the sequential action of index-row operations on the higher/complementary inequality in Lemma~\ref{Add-thm-3-lemma} to conclude Theorem~\ref{Add-thm-3}:
\begin{proof}[Proof of Theorem~\ref{Add-thm-3}]
Suppose $P,Q\subseteq [n]$ such that $|P|+|Q|=n,$ and $Q_1\subseteq P_1$ and $P_2\subseteq Q_2,$ where $P=P_1\sqcup P_2$ and $Q=Q_1\sqcup Q_2$ with $P_1<P_2$ and $Q_1<Q_2.$ Define
\begin{align*}
F_{(X,Y)}^l(A):=(-1)^{1+l}\sum_{k=0}^{l} (-1)^{1+k} \det A\big{(}X\big{|}J_{dk}\big{)} \det A\big{(}Y\big{|}[n]\setminus J_{dk}\big{)},
\end{align*}
for all $n\times n$ TN matrices $A,$ and for sets $X,Y\subseteq [n]$ such that $|X|=d$ and $|Y|=n-d.$ We want to show that $F_{(P,Q)}^l(A)\geq 0$ for all $n\times n$ TN matrices $A.$ To prove this, we use Lemma~\ref{Main-lemma-3} on the inequality $F_{\big{(}[1,d],[d+1,n]\big{)}}^l(A)\geq 0$ for all $n\times n$ TN matrices $A$ -- which we obtained in Lemma~\ref{Add-thm-3-lemma}. 
Suppose each of $P_1,P_2,Q_1,Q_2$ is nonempty, and let 
\begin{align*}
P_1&=\{p_{11}<\cdots<p_{1,p_1}\},~~P_2=\{p_{21}<\cdots<p_{2,p_2}\}, \mbox{ and}\\
Q_1&=\{q_{11}<\cdots<q_{1,q_1}\},~~Q_2=\{q_{21}<\cdots<q_{2,q_2}\}.
\end{align*}
Let $\alpha_1:[p_1+1,d] \to [n-p_2+1,n]$ be the unique order preserving map. For $\alpha_1,$ define:
\begin{align*}
\Ro_{\A_1,j}&:= \Ro_{(\alpha_1(j)-1,\alpha_1(j))}\circ \cdots \circ \Ro_{(j,j+1)}, \mbox{ for all }j\in [p_1+1,d], \mbox{ and}\\
\Ro_{\A_1}&:= \Ro_{\A_1,p_1+1}\circ \cdots\circ\Ro_{\A_1,d}.
\end{align*}
Now let $\beta_1: [d+1,d+q_1] \to [1,q_1]$ be the unique order preserving map. For $\beta_1,$ define:
\begin{align*}
\Ro_{\B_1,j}&:= \Ro_{(\beta_1(j)+1,\beta_1(j))}\circ \cdots \circ \Ro_{(j,j-1)}, \mbox{ for all }j\in [d+1,d+q_1], \mbox{ and}\\
\Ro_{\B_1}&:= \Ro_{\B_1,d+q_1}\circ \cdots\circ\Ro_{\B_1,d+1}.
\end{align*}
It is given that $P_2\subseteq Q_2,$ and $Q_1\subseteq P_1.$ Therefore let 
\begin{align*}
P_2\cap Q_2=\{q_{2,k_1}<\cdots<q_{2,k_{p_2}}\}, \mbox{ and }
Q_1\cap P_1=\{p_{1,l_1}<\cdots<p_{1,l_{q_1}}\}.
\end{align*}
Now, let $\alpha_2:[n-p_2+1,n]\to\{k_1+d+q_1,\cdots, k_{p_2}+d+q_1\}$ be the unique order preserving map. For $\alpha_2,$ define:
\begin{align*}
\Ro_{\A_2,j}&:= \Ro_{(\A_2(j)+1,\A_2(j))}\circ \cdots \circ \Ro_{(j,j-1)}, \mbox{ for all }j\in [n-p_2+1,n], \mbox{ and}\\
\Ro_{\A_2}&:= \Ro_{\A_2,n}\circ \cdots\circ\Ro_{\A_2,n-p_2+1}.
\end{align*}
Once again, let $\beta_2:[1,q_1] \to \{l_1,\cdots,l_{q_1}\}$ be the unique order preserving map. For $\beta_2,$ define:
\begin{align*}
\Ro_{\B_2,j}&:= \Ro_{(\B_2(j)-1,\B_2(j))}\circ \cdots \circ \Ro_{(j,j+1)}, \mbox{ for all }j\in [1,q_1], \mbox{ and}\\
\Ro_{\B_2}&:= \Ro_{\B_2,1}\circ \cdots\circ\Ro_{\B_2,q_1}.
\end{align*}
Let $\A_3:[d+q_1+1,n]\to Q_2$ be the unique order preserving map. For this map, define:
\begin{align*}
\Ro_{\A_3,j}&:= \Ro_{(\A_3(j)+1,\A_3(j))}\circ \cdots \circ \Ro_{(j,j-1)}, \mbox{ for all }j\in [d+q_1+1,n], \mbox{ and}\\
\Ro_{\A_3}&:= \Ro_{\A_3,n}\circ \cdots\circ\Ro_{\A_3,d+q_1+1}.
\end{align*}
Finally, let $\beta_3:[1,p_1]\to P_1$ be the unique order preserving map. For this, define:
\begin{align*}
\Ro_{\B_3,j}&:= \Ro_{(\B_3(j)-1,\B_3(j))}\circ \cdots \circ \Ro_{(j,j+1)}, \mbox{ for all }j\in [1,q_1], \mbox{ and}\\
\Ro_{\B_3}&:= \Ro_{\B_3,1}\circ \cdots\circ\Ro_{\B_3,p_1}.
\end{align*}
Observe that
\begin{align*}
\Ro_{\beta_3} \circ \Ro_{\A_3}\circ \Ro_{\beta_2} \circ \Ro_{\A_2}\circ \Ro_{\beta_1} \circ \Ro_{\A_1}\Big{(}\big{(}[1,d],[d+1,n]\big{)}\Big{)}= \Big{(}\big{(}P,Q\big{)}\Big{)}.
\end{align*}
This completes the proof when each of $P_1,P_2,Q_1,Q_2$ are nonempty. The proof is similar otherwise, which we leave for the reader to complete.
\end{proof}

Note that the class of $P$ and $Q$ stated in Theorem~\ref{Add-thm-3} identifies all possible $\big{(}(P,Q)\big{)}$ (including $P\subseteq Q$ and $Q\subseteq P)$ that can be obtained via operating sequences of index-row operations on $\big{(}([1,d],[d+1,n])\big{)}.$ Therefore we obtained all possible inequalities that can be obtained from Lemma~\ref{Add-thm-3-lemma} via Lemma~\ref{Main-lemma-1} / index-row operations.

\begin{remark}[Index-row operations, Lemma~\ref{Add-thm-3-lemma}, and Theorem~\ref{Add-thm-3}]\label{Add-thm-3-remark}

The proof of Lemma~\ref{Add-thm-3-lemma} (which is $P=Q^{\mathsf{c}}=[1,d]$ in Theorem~\ref{Add-thm-3}) uses the Lindstr\"{o}m Lemma~\ref{Lin-lemma} and judicially manages the weights in the planar network in Figure~\ref{Planar-bw}. This intricate calculation could be worse if we try to use the same planar network (via the Lindstr\"{o}m Lemma~\ref{Lin-lemma}) to prove the inequality for other values of $P$ and $Q$ in Theorem~\ref{Add-thm-3}. However, these represent a fair and straightforward deduction via index-row operations $\Ro_{(u,v)}$ acting on the inequality~\eqref{Add-thm-3-lemma-ineq} in Lemma~\ref{Add-thm-3-lemma}. This particularly signifies the importance of these index-row/column operations.
\end{remark}

\begin{remark}[Conformity and contrasts in Theorems~\ref{Add-thm-1},\ref{Add-thm-2}, and \ref{Add-thm-3}]\label{Confo-cont}
It would not be difficult to deduce the Gantmacher--Krein inequalities \eqref{GK-ineq} and parts of refinements in Theorem~\ref{Add-thm-1} as a consequence of Theorem~\ref{Add-thm-3}. Also, note that the Karlin inequalities in Theorem~\ref{Add-thm-2} differ from the inequalities in Theorem~\ref{Add-thm-3}, as it does not seem obvious how to deduce one from the other.
\end{remark}

\begin{remark}[$\Co_{(u,v)}$ acting on Theorems~\ref{Add-thm-1},\ref{Add-thm-2}, and \ref{Add-thm-3}]\label{Column-op-rem}
In Theorems~\ref{Add-thm-1},\ref{Add-thm-2}, and \ref{Add-thm-3}, we only explored the action of index-row operations and it remains to investigate the action of index-column operations on the inequalities. We expect to obtain novel and nontrivial inequalities by applying the $\Co_{(u,v)}$ on inequalities in Theorem~\ref{Add-thm-3}. We leave this for the reader to complete.
\end{remark}

\begin{remark}[Hierarchy among inequalities]\label{Hierarchy-remark}
As noted in Remark~\ref{Hierarchy-rem-0} (and recalled in Remark~\ref{Hierarchy-rem-1}), we showed that the proofs of Theorems~\ref{Add-thm-1} and \ref{Add-thm-3} reveal a certain {hierarchy} among inequalities. More specifically, we saw that using compositions of index-row operations $\Ro_{(u,v)}$ all the cases of $i\neq j$ can be derived from any of the higher/complementary $i=j$ cases in Theorem~\ref{Add-thm-1}. However, the reverse is not possible. Similarly, all the cases when $P\cap Q\neq \emptyset$ in Theorem~\ref{Add-thm-3} can be derived from the higher/complementary inequality when $P=Q^{\mathsf{c}}=[1,d],$ and the reverse, here as well, is not possible. Finally, none of the inequalities in Theorem~\ref{Add-thm-2} are complementary and, owing to Lemma~\ref{Main-lemma-3}, each of these can be derived from the other. This provided us with an inequality in which this hierarchy goes in the \textit{other} direction as well.

These comments will also be relevant in the next section (see Remark~\ref{FGJ-Skan-comp-remark}) and in Section~\ref{Section-conclusion} (see the Barrett--Johnson inequalities in Theorem~\ref{Skan-Sos} and Question~\ref{Q3}).



\end{remark}

\section{Multiplicative inequalities: main results}\label{Section-multi-results}

In this section we discuss the applicability of the operations $\So_{(u,v)},$ $\Ro_{(u,v)},$ and $\Co_{(u,v)}$ to inequalities for totally nonnegative matrices that are multiplicative. Suppose $n,l\geq 1$ are integers. A multiplicative inequality involving two collections, each consisting of $l$ minors, is of the following form:
\begin{align}\label{Multi-ineq-1}
\prod_{k=1}^{l}\det A \big{(}P_k\big{|}Q_k\big{)} \leq \prod_{k=1}^{l}\det A \big{(}I_k\big{|}J_k\big{)}, \mbox{ for all } n\times n \mbox{ TN matrices }A,
\end{align}
where $P_k,Q_k,I_k,J_k$ are nonempty subsets of $[n]$ with each $|P_k|=|Q_k|$ and $|I_k|=|J_k|.$ It is well known that basic conditions for the inequalities in \eqref{Multi-ineq-1} to hold are the following set of (necessary) multiplicity conditions:

\begin{defn}\label{Multi-ineq-defn-1}
Suppose $n,l\geq 1$ are integers, and let $P_{1},\cdots ,P_l \subseteq [n].$ Define:
\begin{align*}
\ml_P(u):=\big{|}\{k\in [l]: u\in P_k\}\big{|}, \mbox{ for all }u\in [n]. 
\end{align*}
\end{defn}

\begin{lemma}[Fallat--Gekhtman--Johnson \cite{FGJ03}]\label{FGJ-Multi-lemma}
Let $n,l\geq 1$ be integers, and inequality \eqref{Multi-ineq-1} holds. Then,
\begin{align*}
\ml_P(u)=\ml_I(u)\mbox{ and }~\ml_Q(u)=\ml_J(u),\mbox{ for all }u\in [n].
\end{align*}
\end{lemma}

Lemma~\ref{FGJ-Multi-lemma} is not sufficient to classify all the multiplicative inequalities, and therefore further necessary conditions are required to be found. The discussion in and around Lemma~\ref{Main-lemma-1} indicates that when the index-row and index-column operations, $\Ro_{(u,v)}$ and $\Co_{(u,v)},$ operate on an inequality for totally nonnegative matrices, they return an inequality for these matrices. This naturally gives rise to an algorithm to detect certain determinantal expressions that do not form an inequality for these matrices. For this, suppose we call the left and right hand sides of the inequalities in \eqref{Multi-ineq-1} are respectively called the lower and the higher sides. Now, if we choose consecutive $u,v\in [n]$ carefully such that $\Ro_{(u,v)}$ alters the lower side of \eqref{Multi-ineq-1}, then it is necessary for $\Ro_{(u,v)}$ to alter the higher side as well. This leads to another list of multiplicity conditions beyond those identified in Lemma~\ref{FGJ-Multi-lemma}:

\begin{defn}
Suppose $n,l\geq 1$ are integers, and let $P_{1},\cdots ,P_l \subseteq [n].$ Define:
\begin{align*}
\ml_{P}(u,v)&:=|\{k\in [l]:u\in P_k \mbox{ and }v\not\in P_k \}|,\mbox{ for all consecutive }u,v\in [n].
\end{align*}
\end{defn}

\begin{ulemma}\label{Multi-lemma-1}
Let $n,l\geq 1$ be integers, and inequality \eqref{Multi-ineq-1} holds. Then,
\begin{align*}
\ml_{P}(u,v)\leq \ml_I(u,v) \mbox{ and }\ml_{Q}(u,v)\leq \ml_J(u,v),\mbox{ for all consecutive }u,v\in [n].
\end{align*}
\end{ulemma}
\begin{proof}
Suppose, for contradiction, $\ml_{P}(u,v) > \ml_I(u,v),$ for some consecutive integers $u,v\in [n].$ Apply index-row operation $\Ro_{(u,v)}$ on \eqref{Multi-ineq-1}. This gives the  inequality 
\begin{align}\label{Multi-lemma-1-ineq}
\prod_{k=1}^{l}\det A \big{(}P_k(u,v)\big{|}Q_k\big{)} \leq 0, \mbox{ for all } n\times n \mbox{ TN matrices }A.
\end{align}
Expression~\eqref{Multi-lemma-1-ineq} can not hold over TN matrices. Therefore $\ml_{P}(u,v) \leq \ml_I(u,v)$ for all consecutive $u,v\in [n].$ Similarly we can prove the other inequality by applying $\Co_{(u,v)} $ instead.
\end{proof}

(See Section~\ref{minmax} for a generalization of Lemmas~\ref{FGJ-Multi-lemma} and \ref{Multi-lemma-1} that applies to a broader class of inequalities over TN matrices.) From Lemma~\ref{Multi-lemma-1} we deduce the following possibilities: consider inequality \eqref{Multi-ineq-1}, and let $u,v\in [n]$ be consecutive such that $$\So_{(u,v)}\big{(}(P_1,\cdots,P_l)\big{)}\neq \big{(}(P_1,\cdots,P_l)\big{)}.$$ Now apply $\Ro_{(u,v)}$ on inequality \eqref{Multi-ineq-1}. According to Lemma~\ref{Multi-lemma-1}, there exist two cases:
\begin{itemize}
\item[$(a)$] If $\ml_{P}(u,v)=\ml_{I}(u,v),$ then the resulting inequality is,
\begin{align*}
\prod_{k=1}^{l}\det A \big{(}P_k^{1}\big{|}Q_k\big{)} \leq \prod_{k=1}^{l}\det A \big{(}I_k^{1}\big{|}J_k\big{)}, \mbox{ for all } n\times n \mbox{ TN matrices }A,
\end{align*}
where $P_k^1:=P_k(u,v)$ and $I_k^1:=I_k(u,v),$ for all $k\in [l].$
\item[$(b)$] Otherwise, if $\ml_{P}(u,v) < \ml_{I}(u,v),$ then the resulting inequality is,
\begin{align}\label{case-b}
0 \leq \prod_{k=1}^{l}\det A \big{(}I_k^{1}\big{|}J_k\big{)}, \mbox{ for all } n\times n \mbox{ TN matrices }A.
\end{align}
where $I_k^1:=I_k(u,v),$ for all $k\in [l].$
\end{itemize}
We will see when we discuss Algorithm B that \eqref{case-b} is of little use since the resulting inequality is trivial. Before moving ahead, we need the following:


\begin{defn}[Integer pairs]
Let $n\geq 1$ be an integer: 
\begin{itemize}
\item[$(1)$] A pair $(u,v)$ is called an integer pair in $[n]$ provided $u,v\in [n].$
\item[$(2)$] An integer pair $(u,v)$ in $[n]$ is called consecutive provided $u \in \{v-1,v+1\}.$
\end{itemize}
\end{defn}

\begin{ualgorithm}\label{algo}
One of the main problems in total nonnegativity is to complete the classification of all multiplicative inequalities \eqref{Multi-ineq-1}. Lemma~\ref{Multi-lemma-1} and the index-row/column operations are useful in identifying multiplicative determinantal expressions which do not form an inequality. To discover the relevant conditions for this, suppose $n\geq 1$ is an integer and consider, for contradiction, that the following multiplicative inequality holds for nonempty $P_k^0,Q_k^0,I_k^0,J_k^0\subseteq [n]:$ 
\begin{align}\label{Multi-ineq-2}
\prod_{k=1}^{l}\det A \big{(}P_k^{0}\big{|}Q_k^{0}\big{)} \leq \prod_{k=1}^{l}\det A \big{(}I_k^{0}\big{|}J_k^{0}\big{)}, \mbox{ for all } n\times n \mbox{ TN matrices }A.
\end{align}
The success of the algorithm depends on the existence of consecutive integer pairs $(u_1,v_1),$ $\ldots,$ $(u_K,v_K)$ such that the following steps are followed:
\begin{itemize}
\item[$(1)$] Let $s=r=0.$
\item[$(2)$] Let $s=r.$ If there exists consecutive integers $u_{s+1},v_{s+1}
\in [n]$ such that 
\begin{align}\label{algo-cond-1}
\ml_{P^s}(u_{s+1},v_{s+1})>\ml_{I^s}(u_{s+1},v_{s+1}),
\end{align}
then note that inequality \eqref{Multi-ineq-2} does not hold, and break. Otherwise, go to step $(3).$
\item[$(3)$] Suppose there exits consecutive integers $u_{s+1},v_{s+1}\in [n]$ such that
\begin{align}\label{algo-cond-2}
\So_{(u_{s+1},v_{s+1})}\big{(}(P_1^s,\cdots,P_l^s)\big{)}&\neq \big{(}(P_1^s,\cdots,P_l^s)\big{)},\mbox{ and}\nonumber\\
\ml_{P^s}(u_{s+1},v_{s+1})&=\ml_{I^s}(u_{s+1},v_{s+1}).
\end{align}
Apply the index-row operation $\Ro_{(u_{s+1},v_{s+1})}$ to produce
\begin{align*}
\prod_{k=1}^{l}\det A \big{(}P_k^{s+1}\big{|}Q_k^{0}\big{)} \leq \prod_{k=1}^{l}\det A \big{(}I_k^{s+1}\big{|}J_k^{0}\big{)}, \mbox{ for all } n\times n \mbox{ TN matrices }A.
\end{align*}
\item[$(4)$] Update the counter $r=s+1$ and go to step (2).
\end{itemize}

This algorithm basically searches for a sequence of consecutive integer pairs $(u_1,v_1),\cdots,$ $(u_K,v_K)$ such that the application of each composition of index-row operations 
\begin{align}\label{algo-succeed-row-op}
\Ro_{(u_r,v_r)}\circ \cdots \circ \Ro_{(u_1,v_1)}, \mbox{ for }r\in [1,K],
\end{align}
changes the lower side of the inequality on which it acts. Further if, at some point when $r=K,$ the lower side of the inequality changes \textit{more} that the higher side (which violates Lemma~\ref{Multi-lemma-1}), then we produce a contradiction in the form of the inequality,
\begin{align}\label{contradiction}
\prod_{k=1}^{l}\det A \big{(}P_k^{K+1}\big{|}Q_k^{0}\big{)} \leq 0, \mbox{ for all } n\times n \mbox{ TN matrices }A,
\end{align}
which confirms that \eqref{Multi-ineq-2} does not hold. 

This algorithm also has its column-variant which is obtained by replacing all ``$\Ro_{(*,*)}$'' with ``$\Co_{(*,*)}$,'' all $P$ with $Q,$ and all $I$ with $J.$
\end{ualgorithm}

Let us demonstrate the applicability of Algorithm~\ref{algo} via the following example.

\begin{example}\label{multi-example-1}
We consider the following (as in Skandera \cite{Ska04}):
\begin{align*}
P_1=\{1,2,3,6\},~P_2=\{3,4\} &\mbox{ and } I_1=\{1,3,6\},~I_2=\{2,3,4\},\\
Q_1=\{1,2,4,5\},~Q_2=\{2,5\} &\mbox{ and } J_1=\{1,2,5\},~J_2=\{2,4,5\}.
\end{align*}
Here we show that $\det A\big{(}P_1\big{|}Q_1\big{)}\det A\big{(}P_2\big{|}Q_2\big{)}$ and $\det A\big{(}I_1\big{|}J_1\big{)}\det A\big{(}I_2\big{|}J_2\big{)}$ are not comparable for $6\times 6$ TN matrices, using Algorithm~\ref{algo}:
\begin{itemize}
\item[$(a)$] Suppose the following is an inequality over TN matrices:
$$
\det A\big{(}P_1\big{|}Q_1\big{)}\det A\big{(}P_2\big{|}Q_2\big{)} \geq \det A\big{(}I_1\big{|}J_1\big{)}\det A\big{(}I_2\big{|}J_2\big{)}.
$$ 
Apply index-row operation $\Ro_{(1,2)}$ on above. This yields a false inequality of the form \eqref{contradiction}. This therefore denies the inequality just above.
\item[$(b)$] Similarly, apply index-row operations $\Ro_{(4,3)} \circ \Ro_{(3,2)}$ on the expression 
$$
\det A\big{(}P_1\big{|}Q_1\big{)}\det A\big{(}P_2\big{|}Q_2\big{)} \leq \det A\big{(}I_1\big{|}J_1\big{)}\det A\big{(}I_2\big{|}J_2\big{)}
$$ 
to reach a false inequality of the form \eqref{contradiction}, as in $(a).$
\end{itemize}
See the following remark for further comments.
\end{example}

\begin{remark}[A significant outcome when Algorithm~\ref{algo} succeeds]\label{algo-succeed-row-op-remark}
We should also observe that the two inequalities in Example~\ref{multi-example-1} are denied without using any action on the column indices. This implies that the choices of the row indices $P_1,P_2, I_1, I_2$ are such that product of the minors are never going to be comparable for any choices of the column indices $Q_1,Q_2,J_1,J_2.$ More generally, if Algorithm~\ref{algo} succeeds for \eqref{Multi-ineq-2} via applying index-row operations in ~\eqref{algo-succeed-row-op}, then for all choices of $Q_k^0$ and $J_k^0,$ expression \eqref{Multi-ineq-2} shall not form an inequality over totally nonnegative matrices. A similar conclusion arises for $P_k^0$ and $I_k^0$ if Algorithm~\ref{algo} succeeds for analogous index-column operations in \eqref{algo-succeed-row-op}.
\end{remark}

In addition to the outcome above, Algorithm~\ref{algo} has important flaws that inhibit it from succeeding in all situations.

\begin{remark}[General limitations of Algorithm~\ref{algo}]\label{algo-limit-remark}
Firstly, for Algorithm~\ref{algo} to succeed, we need to ensure that the two prerequisites in step (3) are satisfied simultaneously for consecutive integers $u_{s+1}$ and $v_{s+1}.$ Otherwise we will tend towards a \eqref{case-b} scenario, which is not productive in this detection process. Thus, conditions in step (3) seem minimal. That being said, this condition is not easily satisfied when $l\geq 3,$ and always holds true for $l=2$ owing to the multiplicity Lemmas~\ref{FGJ-Multi-lemma} and~\ref{Multi-lemma-1}. Secondly, the larger issue with this algorithm is that the operations we apply on the row indices show little connection with the corresponding column indices, despite the fact that the given algorithm can also be run (simultaneously or otherwise) on the column indices. Clearly for this reason, the algorithm is not sufficient to deduce all of the desired multiplicative inequalities. 

The second shortcoming can be avoided by finding an \textit{updated} form of the indices which \textit{nicely} incorporates the row and column indices together. If we restrict ourselves to $l=2$ in \eqref{Multi-ineq-1} then a slightly updated algorithm would be sufficient as well, which we discuss next.
\end{remark}

\subsection*{Smallest multiplicative inequalities} 
We call the multiplicative inequalities in \eqref{Multi-ineq-1} as the smallest multiplicative inequalities whenever $l=2.$ These smallest inequalities over $n\times n$ TN matrices are classified in terms of the smallest multiplicative inequalities, consisting only of the principal minors, over $2n \times 2n$ TN matrices. This correspondence is the key which provides us with the required \textit{combination} of the row and column indices (which we alluded to in the previous paragraph) for Algorithm~\ref{algo} to succeed completely. The precise result containing the correspondence is the following:

\begin{defn}[Consecutive, contiguous, and even]\label{cce-defn}
Henceforth we will use the following terms relative to the underlying set that will be clear from context. Let $n,k\geq 1$ be integers, and suppose $M:=\{m_1 < \cdots < m_k\} \subseteq [n].$ 
\begin{itemize}
\item[$(1)$] Integers $u,v$ are called {\em consecutive in} $M$ if $u=m_l$ and $v=m_{l+1},$ or vice versa, for some $l\in [1,k-1].$
\item[$(2)$] A subset $S\subseteq M$ is called {\em contiguous in} $M$ if there exists an interval $L \subseteq [k]$ such that $S=\{m_l:l\in L\}.$
\item[$(3)$] A subset $S\subseteq M$ is \textit{even} if the number of elements in $S$ is even. 
\item[$(4)$] For every integer $l,$ $~l- M:=\{l- m:m\in M\}.$
\end{itemize}
\end{defn}

\begin{theorem}[Fallat--Gekhtman--Johnson \cite{FGJ03}, Skandera \cite{Ska04}]\label{FGJ-Skan-thm}
Let $n\geq 1$ be an integer and $P_1, P_2, Q_1, Q_2, I_1, I_2, J_1, J_2 \subseteq [n]$ be nonempty with each $|P_k|=|Q_k|$ and $|I_k|=|J_k|.$ Define
\begin{align*}
R_1&:=P_1\cup (2n+1-Q_2),~R_2:=P_2 \cup (2n+1-Q_1), \\
K_1&:=I_1\cup (2n+1-J_2),~K_2:=I_2 \cup (2n+1-J_1), \mbox{ and}\\
M&:= (R_1\setminus R_2)\cup (R_2\setminus R_1).
\end{align*}
If $|P_1|+|P_2|=|I_1|+|I_2|=n,$ then the following are equivalent:
\begin{itemize}
\item[$(1)$] For all $n\times n$ totally nonnegative matrices $A,$
\begin{align}\label{smallest-multi}
\det A\big{(}P_1\big{|}Q_1\big{)}\det A\big{(}P_2\big{|}Q_2\big{)} \leq \det A\big{(}I_1\big{|}J_1\big{)}\det A\big{(}I_2\big{|}J_2\big{)}.
\end{align}
\item[$(2)$] For all $2n\times 2n$ totally nonnegative matrices $B,$
\begin{align}\label{FGJ-condition}
\det B\big{(}R_1\big{|}R_1\big{)}\det B\big{(}R_2\big{|}R_2\big{)} \leq \det B\big{(}K_1\big{|}K_1\big{)}\det B\big{(}K_2\big{|}K_2\big{)}.
\end{align}
\item[$(3)$] Multisets $R_1\cup R_2$ and $K_1\cup K_2$ are equal, and for all even contiguous $S \subseteq M,$
\begin{align}\label{Skandera-condition}
\max\{|S\cap R_1|,|S \cap R_2|\}\geq \max\{|S\cap K_1|,|S\cap K_2|\}.
\end{align}
\end{itemize}
\end{theorem} 

We reasoned in Remark~\ref{algo-limit-remark} that in order to use the index-row/column operations in an effort to classify multiplicative inequalities, we need to apply these operations on a mix of row and column indices. For the smallest multiplicative inequalities, Theorem~\ref{FGJ-Skan-thm} provides us with one such mix, namely $(R_1, R_2)$ and $(K_1,K_2),$ and we have obtained an upgraded algorithm that can, in fact, detect all the multiplicative expressions that do not form an inequality:

\begin{utheorem}\label{Multi-thm-1}
The following is equivalent to (1), (2), and (3) in Theorem~\ref{FGJ-Skan-thm}:
\vspace*{1mm}
\begin{itemize}
\item[$(4)$] Multisets $R_1\cup R_2$ and $K_1\cup K_2$ are equal, and for all consecutive integer pairs $(u_1,v_1),\cdots,(u_k,v_k)$ in $[2n],$ for all integers $k\geq 1,$
\begin{align*}
\mbox{whenever }~\So_{(U_j,V_j)} \big{(}(R_1,R_2)\big{)} &\neq \So_{(U_{j-1},V_{j-1})} \big{(}(R_1,R_2)\big{)},\\
\mbox{then }~\So_{(U_j,V_j)} \big{(}(K_1,K_2)\big{)} &\neq \So_{(U_{j-1},V_{j-1})} \big{(}(K_1,K_2)\big{)},
\end{align*}
where, $\So_{(U_{0},V_{0})}:=\So_{(1,1)},$ and 
$\So_{(U_{j},V_{j})}:=\So_{(u_{j},v_{j})}\circ\So_{(U_{j-1},V_{j-1})}$ for all $j\in [1,k].$
\end{itemize}
\end{utheorem}

Theorem~\ref{Multi-thm-1} states that a smallest multiplicative inequality is valid over TN matrices if and only if Algorithm~\ref{algo} fails for the given slightly modified expression provided in \eqref{FGJ-condition}. 

It is now worthwhile and appropriate to recall Remark~\ref{MainLemma-converse} where we discussed Lemma~\ref{Main-lemma-1} and pointed out that compositions of the index-row and index-column operations $\Ro_{(u,v)}$ and $\Co_{(u,v)}$ acting on inequalities over TN matrices return inequalities for TN matrices, and the converse of this statement seems too good to be true in general. However, for the class of the smallest multiplicative inequalities the converse is indeed true, though, in a very slight modified sense. If needed, review Remark~\ref{MainLemma-action} before we state the precise result below:

\begin{utheorem}\label{Multi-thm-2}
The following is equivalent to (1), (2), and (3) in Theorem~\ref{FGJ-Skan-thm}:
\begin{itemize}
\item[$(5)$] Multisets $R_1\cup R_2$ and $K_1\cup K_2$ are equal, and for all consecutive integer pairs $(u_1,v_1),\ldots,(u_k,v_k)$ in $[2n],$ for all integers $k\geq 1,$ the index-row operation
\begin{align}\label{Multi-thm-2-ineq}
\Ro_{(u_{k},v_{k})} \circ \cdots \circ \Ro_{(u_1,v_1)}
\end{align}
applied on the determinantal expression in \eqref{FGJ-condition} in Theorem~\ref{FGJ-Skan-thm} returns an expression which forms a valid inequality over all $2n\times 2n$ totally nonnegative matrices.
\end{itemize}
This result also has the ``mixed'' row/column-variant which we obtain by replacing any of the ``$\Ro_{(*,*)}$''s with ``$\Co_{(*,*)}$''s in \eqref{Multi-thm-2-ineq}.
\end{utheorem}

\begin{remark}[Significance of Theorems~\ref{Multi-thm-1} and \ref{Multi-thm-2}]\label{Multi-thm-2-significance}
As discussed earlier, Theorem~\ref{Multi-thm-2} can be viewed as the converse of Lemma~\ref{Main-lemma-1} for the smallest multiplicative inequalities. Because of the universality of Lemma~\ref{Main-lemma-1}, Theorem~\ref{Multi-thm-2} seems a very natural classification to consider for all the smallest multiplicative inequalities. Hence Theorem~\ref{Multi-thm-2} is a constructive and worthy addition to the list of equivalent conditions in the classification Theorem~\ref{FGJ-Skan-thm} due to Fallat--Gekhtman--Johnson \cite{FGJ03} and Skandera \cite{Ska04} {(which was also later revisited by Rhoades--Skandera \cite{RS2005, RS2006})}. 

Note that the action of index-row operations $\Ro_{(*,*)}$ does not impede any such index-column operation $\Co_{(*,*)},$ and vice versa. Consequently, proof of Theorem~\ref{Multi-thm-2} is a straightforward deduction from Theorem~\ref{Multi-thm-1}, which we leave for the reader to complete.
\end{remark}

The class of smallest multiplicative inequalities for TN matrices is fairly large and complicated. For pragmatic purposes therefore, it is reasonable to know different ways to identify these inequalities. Algorithm~\ref{algo} provides other such ways, and we show this via Theorem~\ref{Multi-thm-1} and Theorem~\ref{Multi-thm-2}. Let us demonstrate the applicability of these via the example below.

\begin{example}\label{multi-example-2}
We consider the following:
\begin{align*}
P_1=\{1,3,4\},~P_2=\{2,5,6\} &\mbox{ and } I_1=\{1,3,4\},~I_2=\{2,5,6\},\\
Q_1=\{1,2,3\},~Q_2=\{4,5,6\} &\mbox{ and } J_1=\{1,2,4\},~J_2=\{3,5,6\}.
\end{align*}
Here we show that $\det A\big{(}P_1\big{|}Q_1\big{)}\det A\big{(}P_2\big{|}Q_2\big{)}$ and $\det A\big{(}I_1\big{|}J_1\big{)}\det A\big{(}I_2\big{|}J_2\big{)}$ are not comparable for $6\times 6$ TN matrices. This is equivalent to showing that $\det B\big{(}R_1\big{|}R_1\big{)}\det B\big{(}R_2\big{|}R_2\big{)}$ and $\det B\big{(}K_1\big{|}K_1\big{)}\det B\big{(}K_2\big{|}K_2\big{)}$ are not comparable for $12\times 12$ TN matrices, where
\begin{align*}
R_1&:= \{1,3,4,7,8,9\},~R_2:=\{2,5,6,10,11,12\}, \\
K_1&:= \{1,3,4,9,11,12\},~K_2:=\{2,5,6,7,8,10\}.
\end{align*}
\begin{itemize}
\item[$(a)$] We leave it to the reader to verify and convince themselves that 
$$
\det A\big{(}P_1\big{|}Q_1\big{)}\det A\big{(}P_2\big{|}Q_2\big{)} \leq \det A\big{(}I_1\big{|}J_1\big{)}\det A\big{(}I_2\big{|}J_2\big{)}
$$ 
cannot be falsified by directly applying Algorithm~\ref{algo} on it. In other words, Algorithm~\ref{algo} fails to detect that this expression does not form an inequality (unlike Example~\ref{multi-example-1}). This limitation is also outlined in Remark~\ref{algo-limit-remark} and, hence, this is where Theorems~\ref{Multi-thm-1} and~\ref{Multi-thm-2} are required to be applied. 

\item[$(b)$] Following Theorem~\ref{Multi-thm-2}, if we apply $\Ro_{(6,7)}$ on the equivalent
\begin{align}\label{example-ineq}
\det B\big{(}R_1\big{|}R_1\big{)}\det B\big{(}R_2\big{|}R_2\big{)} \leq \det B\big{(}K_1\big{|}K_1\big{)}\det B\big{(}K_2\big{|}K_2\big{)}
\end{align}
then we can obtain a false inequality of the form \eqref{contradiction}, which therefore denies the inequality just above and hence the one in part $(a).$

\item[$(c)$] There are various ways (similar to Example~\ref{multi-example-1}) to deny 
$$
\det A\big{(}P_1\big{|}Q_1\big{)}\det A\big{(}P_2\big{|}Q_2\big{)} \geq \det A\big{(}I_1\big{|}J_1\big{)}\det A\big{(}I_2\big{|}J_2\big{)}.
$$ 
For example, apply $\Co_{(2,3)}$ or $\Co_{(4,5)}$ to obtain a false inequality of the form \eqref{contradiction}.
\end{itemize}
Example~\ref{multi-example-1} and part $(c)$ in this example show the importance of the Algorithm~\ref{algo}, and that we need not always refer to Theorems~\ref{FGJ-Skan-thm} or \ref{Multi-thm-2}. Part $(a),$ on the other hand, signifies the value of Theorems~\ref{FGJ-Skan-thm} and \ref{Multi-thm-2} over Algorithm~\ref{algo}. However, we also understand that even here, owing to Theorem~\ref{Multi-thm-2}, Algorithm~\ref{algo} is used for the revised determinantal expression \eqref{example-ineq} in part $(b).$
\end{example}

\begin{remark}[Hierarchy among inequalities]\label{FGJ-Skan-comp-remark}
Recall Remarks~\ref{Hierarchy-rem-0} and \ref{Hierarchy-remark} that the inequalities in Theorems~\ref{Add-thm-1} and \ref{Add-thm-3} contain a certain hierarchy, and each of these inequalities are derivative of the complimentary ones via index-row/column operations. Thus, we naturally follow up from the smallest multiplicative inequalities, and find that all these too are derivative of the complementary smallest multiplicative inequalities. This entails proving the rather challenging ``converse'' that all of the smallest multiplicative inequalities are derivable via applying sequential index-row/column operations on complementary smallest multiplicative inequalities.

\end{remark}

\begin{utheorem}\label{Multi-thm-3}
The following is equivalent to (1), (2), and (3) in Theorem~\ref{FGJ-Skan-thm}:
\begin{itemize}
\item[$(6)$] There exist $P,Q,I,J \subseteq [n]$ with $|P|=|Q|=|P_1|$ and $|I|=|J|=|I_1|$ such that, for all $n\times n$ totally nonnegative matrices $A,$
\begin{align}\label{Multi-thm-3-ineq-2}
\det A\big{(}P\big{|}Q\big{)}\det A\big{(}P^{\mathsf{c}}\big{|}Q^\mathsf{c}\big{)} \leq \det A\big{(}I\big{|}J\big{)}\det A\big{(}I^{\mathsf{c}}\big{|}J^{\mathsf{c}}\big{)}.
\end{align}
Moreover, there exist integers $k_r,k_c \geq 1,$ and integer pairs $(u_i,v_i)$ and $(u_j',v_j')$ in $[n],$ for $i \in [1,k_r]$ and $j\in [1,k_c],$ where each $u_i\in\{v_i-1,v_i,v_i+1\}$ and each $u_j'\in\{v_j'-1,v_j',v_j'+1\},$ such that the composite operation,
\begin{align*}
\Ro_{(u_{k_r},v_{k_r})} \circ \cdots \circ \Ro_{(u_1,v_1)}\circ
\Co_{(u_{k_{c}}',v_{k_{c}}')} \circ \cdots \circ \Co_{(u_1',v_1')}
\end{align*}
applied on inequality \eqref{Multi-thm-3-ineq-2} results in inequality \eqref{smallest-multi} in Theorem~\ref{FGJ-Skan-thm}.
\end{itemize}
\end{utheorem}

\begin{remark}[Significance of Theorem~\ref{Multi-thm-3}]\label{Multi-thm-3-significance}
For any given smallest multiplicative inequality \eqref{smallest-multi} in Theorem~\ref{FGJ-Skan-thm}, earlier we knew only of the inequality in \eqref{FGJ-condition} that would exist as an equivalent presence. Now, because of the index-row/column operations, we have a much bigger class of smallest multiplicative inequalities that exists as well. In fact, Theorem~\ref{Multi-thm-3} shows that this existence is equivalent to the existence of the given inequality in the first place. This is ensured because Theorem~\ref{Multi-thm-3} establishes that in order to hold on to all of the smallest multiplicative inequalities, we are only required to identify the higher/complementary ones. This also provides another description of the smallest multiplicative inequalities, which augments Theorems~\ref{FGJ-Skan-thm}, \ref{Multi-thm-1}, and \ref{Multi-thm-2}, and thus would have its own applicability in the theory of multiplicative inequalities over totally nonnegative matrices.
\end{remark}

\section{Multiplicative inequalities: proofs}\label{Multi-proofs}


The proof of Theorem~\ref{Multi-thm-1} requires the following:

\begin{lemma}\label{FGJ-Skan-review-lemma} Let $n\geq 1$ be an integers, and $R_1,$ $R_2,$ $K_1,$ and $K_2$ be nonempty subsets of $[2n].$ Define $M=(R_1\setminus R_2)\cup (R_2\setminus R_1),$ and suppose multisets $R_1\cup R_2$ and $K_1\cup K_2$ are equal, and for all even contiguous subsets $S\subseteq M,$
\begin{align}\label{FGJ-Skan-review-lemma-condition}
\max\{|S\cap R_1|,|S \cap R_2|\}\geq \max\{|S\cap K_1|,|S\cap K_2|\}.
\end{align}
Then for all consecutive $u,v\in [n],$
\begin{align*}
\So_{(u,v)}\big{(}(K_1,K_2)\big{)} \neq \big{(}(K_1,K_2)\big{)} \mbox{ whenever } ~\So_{(u,v)}\big{(}(R_1,R_2)\big{)}\neq \big{(}(R_1,R_2)\big{)}.
\end{align*}
\end{lemma}
\begin{proof}
Let $u,v\in [n]$ be consecutive, and suppose 
\begin{align*}
\big{(}(K_1',K_2')\big{)}:=\So_{(u,v)}\big{(}(K_1,K_2)\big{)}\mbox{ and }\big{(}(R_1',R_2')\big{)}:=\So_{(u,v)}\big{(}(R_1,R_2)\big{)}.
\end{align*}
The given multiplicity conditions imply, without loss of any generality, that for $\big{(}(R_1',R_2')\big{)}\neq \big{(}(R_1,R_2)\big{)}$ we must have one of the following two cases:
\begin{align*}
\mbox{either }u\in R_1\cap R_2 \mbox{ and } v\not\in R_1\cup R_2,~~
\mbox{or}~~u\in R_1\setminus R_2 \mbox{ and }v\not\in R_1.
\end{align*}
Now, for a contradiction, suppose $\big{(}(K_1',K_2')\big{)}=\big{(}(K_1,K_2)\big{)}.$ If $u \in R_1 \cap R_2$ and $v\not\in R_1\cup R_2$ then, either $u \not \in K_1 \cup K_2$ or $\{u,v\} \subseteq K_1 \cap K_2.$ This contradicts the given multiplicity condition. If $u\in R_1\setminus R_2$ and $v\not\in R_1$, then, without loss of generality, $\{u,v\}\subseteq (K_1 \setminus K_2)\setminus( K_2 \setminus K_1).$ This means $S=\{u,v\}$ violates \eqref{FGJ-Skan-review-lemma-condition}, which completes the proof.
\end{proof}

Now we prove Theorem~\ref{Multi-thm-1}:

\begin{proof}[Proof of Theorem~\ref{Multi-thm-1}] We prove the equivalence of $(3)$ in Theorem~\ref{FGJ-Skan-thm} and $(4)$ in Theorem~\ref{Multi-thm-1}. Note that $(3)\implies (4)$ follows from Theorem~\ref{FGJ-Skan-thm}, and Lemmas~\ref{FGJ-Skan-review-lemma} and \ref{Main-lemma-1}. We prove $(4)\implies (3):$ suppose $(3)$ does not hold. Then there exists an even contiguous nonempty subset $S\subseteq M$ such that
\begin{align}\label{let1}
\max\{|S\cap R_1|,|S \cap R_2|\} < \max\{|S\cap K_1|,|S\cap K_2|\}.
\end{align}
Without loss of generality, we suppose $S$ is minimal, i.e., for each proper even contiguous subsets $D\subset S \subseteq M$
\begin{align}\label{let2}
\max\{|D\cap R_1|,|D \cap R_2|\} \geq \max\{|D\cap K_1|,|D\cap K_2|\}.
\end{align}
We claim that for $D_S:=S\setminus\{\min S, \max S\},$
\begin{align}\label{claim1}
|D_S\cap R_1|=|D_S \cap R_2|=|D_S\cap K_1|=|D_S\cap K_2|.
\end{align}
If $D_S$ is empty then claim~\eqref{claim1} is obviously true. So suppose $D_S=\{d_1 <\cdots < d_{2l}\},$ for some integer $l\geq 1.$

\begin{uclaim}\label{claim-A}
$D_S$ is nonempty if and only if $D_S \cap R_1,$ $D_S\cap R_2,$ $D_S\cap K_1,$ and $D_S\cap K_2$ are all nonempty.
\end{uclaim}
\begin{proof}
Suppose $D_S$ is nonempty. Then without loss of generality, let $D_S \cap R_1$ and $D_S\cap K_1$ be nonempty. We will show that $D_S \cap R_2$ and $D_S\cap K_2$ are nonempty:

If $D_S\cap R_2=\emptyset,$ then $\max\{|D_S\cap R_1|,|D_S\cap R_2|\}=2l.$ If $\max\{|D_S\cap K_1|,|D_S\cap K_2|\}\leq 2l-1,$ then there exist contiguous $k_1,k_2\in D_S$ such that $k_1\in K_1 \setminus K_2$ and $k_2\in K_2\setminus K_1.$ Since $\max\{|D_S\cap R_1|,|D_S\cap R_2|\}=2l,$ $\max\{|S\cap R_1|,|S\cap R_2|\}\geq 2l.$ On the other hand, since $\max\{|D_S\cap K_1|,|D_S\cap K_2|\}\leq 2l-1,$ $\max\{|S\cap K_1|,|S\cap K_2|\}\leq 2l+1.$ Hence 
\begin{align*}
2l=\max\{|D_S\cap R_1|,|D_S\cap R_2|\} & > \max\{|D_S\cap K_1|,|D_S\cap K_2|\} = 2l-1, \mbox{ and} \\
2l=\max\{|S\cap R_1|,|S \cap R_2|\} &< \max\{|S\cap K_1|,|S\cap K_2|\}=2l+1.
\end{align*}
Since $D_S \neq \emptyset,$ either $S_1:=\{\min S, \min D_S\}$ or $S_2:=\{\max S, \max D_S\}$ which violates the minimality of $S$ in \eqref{let1} and \eqref{let2}. Therefore $\max\{|D_S\cap K_1|,|D_S\cap K_2|\} = 2l.$ Thus $|D_S\cap K_1|=2l.$ The following cases are possible:
\begin{itemize}
\item[$(a)$] $\{\min S, \max S\}\subset K_1:$ Because of \eqref{let1} and \eqref{let2}, we must have either $\min S \not\in R_1$ or $\max S \not \in R_1.$ In that case, either $S_1 = \{\min S, \min D_S\}$ or $S_2:=\{\max S, \max D_{S}\}$ which violates the minimality of $S$ in \eqref{let1} and \eqref{let2}.
\item[$(b)$] Either $\min S \in K_1$ and $\max S \in K_2,$ or vice versa: if $\min S \in K_1$ and $\max S \in K_2,$ then $\min S \in R_1,$ because otherwise the minimality of $S$ is violated as in the previous case. This ensures that \eqref{let1} does not hold for $S,$ which is a contradiction. Similarly we can rule out the other case when $\min S \in K_2$ and $\max S \in K_1.$
\end{itemize}
Thus $D_S\cap R_2$ is nonempty. Finally, because of this and \eqref{let2}, $D_S \cap K_2$ is also nonempty.

The converse is obvious.
\end{proof} 

\begin{uclaim}\label{claim-B}
Suppose $u,v \in [\min D_S,\max D_S]$ are consecutive, and define
\begin{align*}
\big{(}(R_1',R_2'),(K_1',K_2')\big{)}:=\So_{(u,v)}\big{(}(R_1,R_2),(K_1,K_2)\big{)},~~M':= (R_1'\setminus R_2')\cup (R_2'\setminus R_1'),
\end{align*}
\begin{align*}
S'&:=
\begin{cases}
S \setminus \{u,v\} & \mbox{if } u \in R_1 \setminus R_2 \mbox{ and } v \in R_2 \setminus R_1 \mbox{ or vice versa,}\\
S(v,u) & \mbox{if } u \in R_1 \cap R_2, \mbox{ and } v \in R_2 \setminus R_1 \mbox{ or }v \in R_1 \setminus R_2, \mbox{ and} \\
S & \mbox{otherwise,}
\end{cases} \\
D_{S'}&:=S'\setminus \{\min S',\max S'\}.
\end{align*}
If $\big{(}(R_1',R_2')\big{)}\neq \big{(}(R_1,R_2)\big{)},$ then the multisets $R_1'\cup R_2'$ and $K_1'\cup K_2'$ are equal, and the  conditions in $(4)$ in Theorem~\ref{Multi-thm-1} hold for $R_1',R_2',K_1',K_2'.$ Moreover, $S'\subseteq M'$ is minimal for $R_1',R_2',K_1',K_2',$ i.e.,
\begin{align*}
\max\{|S'\cap R_1'|,|S' \cap R_2'|\} < \max\{|S'\cap K_1'|,|S'\cap K_2'|\},
\end{align*}
and for each proper even contiguous subsets $D'\subset S'$
\begin{align*}
\max\{|D'\cap R_1'|,|D' \cap R_2'|\} \geq \max\{|D'\cap K_1'|,|D'\cap K_2'|\}.
\end{align*}
In addition to all this, Claim~\ref{claim-A} holds for $D_{S'}$ and $R_1',R_2',K_1',K_2',$ i.e., $D_{S'}:= S'\setminus\{\min S',\max S'\}$ is nonempty if and only if $D_{S'} \cap R_1',$ $D_{S'}\cap R_2',$ $D_{S'}\cap K_1',$ and $D_{S'}\cap K_2'$ are all nonempty.
\end{uclaim}
\begin{proof}
It is not difficult to see that if $S'$ is not minimal for $R_1',R_2',K_1',K_2',$ then $S$ can not be minimal for $R_1,R_2,K_1,K_2,$ which would be a contradiction to \eqref{let1} and \eqref{let2}. Hence, because of $(4)$ in Theorem~\ref{Multi-thm-1}, all of the remaining conclusions in Claim~\ref{claim-B} hold. 
\end{proof} 

Now that Claims~\ref{claim-A} and \ref{claim-B} are set, we provide an algorithmic proof of \eqref{claim1}. Define:
\begin{align*}
R_1^{0}:=R_1,~R_2^0:=R_2,~K_1^0:=K_1,~K_2^0:=K_2,~S^0:=S,~D_{S^0}:=D_{S}.
\end{align*}
\begin{itemize}
\item[$(a)$] Suppose $r=s=0.$
\item[$(b)$] Suppose $r=s.$ If $D_{S^r}\neq \emptyset$, then there exists consecutive integers $u_r,v_r\in D_{S^r}$ such that $u_r \in R_1^r \setminus R_2^r$ and $v_r \in R_2^r \setminus R_1^r.$ The multiplicity conditions imply that $u^r \in K_1^r\setminus K_2^r$ and $v^r \in K_2^r\setminus K_1^r,$ or vice versa. Suppose $u_{1,r} < u_{2,r} < \cdots < u_{k_r,r} \in [\min\{u_r,v_r\} + 1,\max\{u_r,v_r\}-1]$ such that each $u_{i,r}\in R_1^r\cup R_2^r.$ Note that each $u_{i,r} \in R_1^r\cap R_2^r.$ Suppose $u_{0,r}:=\min\{u_r,v_r\}$ and $u_{{k_r,r}+1}:= \max\{u_r,v_r\},$ and define the operations:
\begin{align*}
\So_{r,j}:= \So_{(u_{0,r}+{j},u_{0,r}+{j-1})}\circ \cdots \circ \So_{(u_{j,r}-1,u_{j,r}-2)} \circ \So_{(u_{j,r},u_{j,r}-1)}, \mbox{ for all }j\in [1,k_r+1].
\end{align*}
Now define
\begin{align*}
\Big{(}(R_1^{r+1},R_2^{r+1}),(K_1^{r+1},K_2^{r+1})\Big{)}&:=\So_{r}\Big{(}(R_1^r,R_2^r),(K_1^r,K_2^r)\Big{)}, \\
\mbox{where, }~~ \So_{r}&:= \So_{r,k_r+1} \circ  \cdots  \circ \So_{r,1}.
\end{align*}
Note that, each $\So_{(*,*)}$ above is chosen in such a way that it changes the values of the sets it operates on. Therefore, Claims~\ref{claim-A} and \ref{claim-B} ensure that $S^{r+1},$ which can be obtained from recursive implementation of Claim~\ref{claim-B} for $\So_r$ acting on $(R_1^r,R_2^r),(K_1^r,K_2^r),$ is minimal for $R_1^{r+1},R_2^{r+1},K_1^{r+1},K_2^{r+1}.$ Define $D_{S^{r+1}}:=S^{r+1}\setminus\{\min S^{r+1}, \max S^{r+1}\},$ and note that $|D_{S^{r+1}}|=|D_{S^{r}}|-2.$
\item[$(c)$] If $D_{S^{r+1}} \neq \emptyset,$ then update the counter $s=r+1$ and go to step (b). Otherwise, break.
\end{itemize}

The algorithm above breaks after $l$ iterations, where $2l=|D_{S^0}|.$ Note that at the end of $l^{\mbox{th}}$ iteration, $D_{S^l} =\emptyset.$ Because of $(4)$ in Theorem~\ref{Multi-thm-1} and Claims~\ref{claim-A} and \ref{claim-B}, and the minimality of $S,$ this is possible only if \eqref{claim1} holds. This, further, implies that in order for \eqref{let1} to hold, we must have exactly one of $\max S$ and $\min S$ in exactly one of $R_1$ and $R_2,$ and either $\{\min S, \max S \} \subseteq K_1$ and  $\{\min S, \max S \} \cap K_2 = \emptyset$ or vice versa.

Now suppose $u_{1,l+1} < \cdots < u_{k_{l+1},l+1}$ are all integers in $[\min S+1,\max S-1]$ such that each $u_{i,l+1} \in R_1^l\cup R_2^l$ (which is equal to $R_1^l\cap R_2^l).$ Define $u_{0,l+1}:=\min S$ and $u_{k_{l+1}+1,l+1}:=\max S,$ and
\begin{align*}
\So_{l+1,j}:= \So_{(u_{0,l+1}+{j},u_{0,l+1}+{j-1})}\circ \cdots \circ \So_{(u_{j,l+1}-1,u_{j,l+1}-2)} \circ \So_{(u_{j,l+1},u_{j,l+1}-1)}, 
\end{align*}
$\mbox{ for all }j\in [1,k_{l+1}+1],$ and define,
$
\So_{l+1}:= \So_{l+1,k_{l+1}+1} \circ  \cdots  \circ \So_{l+1,1}.
$

Note that there exists an integer $k\geq 1$ and consecutive integer pairs $(u_1,v_1),\ldots,$ $(u_k,v_k),$ such that
$$
\So_{(u_k,v_k)}\circ \cdots \circ \So_{(u_1,v_1)}=\So_{l+1}\circ \So_l \circ \cdots \circ \So_2 \circ \So_1.
$$ 
Observe that 
\begin{align*}
\So_{(U_j,V_j)} \big{(}(R_1,R_2)\big{)} &\neq \So_{(U_{j-1},V_{j-1})} \big{(}(R_1,R_2)\big{)}, \mbox{ for all }j\in [1,k],
\end{align*}
where, $\So_{(U_{0},V_{0})}:=\So_{(1,1)}$ and $\So_{(U_{j},V_{j})}:=\So_{(u_{j},v_{j})}\circ\So_{(U_{j-1},V_{j-1})}$ for all $j\in [1,k].$
However, the conclusion in paragraph just below the algorithm implies
\begin{align*}
\So_{(U_k,V_k)} \big{(}(K_1,K_2)\big{)} = \So_{(U_{k-1},V_{k-1})} \big{(}(K_1,K_2)\big{)}.
\end{align*}
This contradicts (4) in Theorem~\ref{Multi-thm-1}. Therefore, (3) in Theorem~\ref{FGJ-Skan-thm} holds. This completes the equivalence of (3) in Theorem~\ref{FGJ-Skan-thm} and (4) in Theorem~\ref{Multi-thm-1}.
\end{proof}

We prove Theorem~\ref{Multi-thm-3} now:

\begin{proof}[Proof of Theorem~\ref{Multi-thm-3}]
We present an algorithmic proof to convert $\big{(}(P_1,P_2),(I_1,I_2)\big{)}$ and $\big{(}(Q_1,Q_2),(J_1,J_2)\big{)}$ into complementary row/column index sets $\big{(}(P,P^\mathsf{c}),(I,I^\mathsf{c})\big{)}$ and $\big{(}(Q,Q^\mathsf{c}),$ $(J,J^\mathsf{c})\big{)},$ and vice versa, respectively:
\begin{itemize}
\item[$(1)$] If $(P_1, P_2)$ and $ (I_1,I_2)$ are complementary, i.e. $P_1=P_2^{\mathsf{c}}$ and $I_1=I_2^{\mathsf{c}},$ let 
\begin{align*}
\Ro:= \Ro_{(1,1)}, \mbox{ and go to step (7).}
\end{align*}
Otherwise, let $X_1:=P_1, X_2:=P_2, Y_1:=I_1$ and $Y_2:=I_2,$ and go to step (2).

\item[$(2)$] Let $N:= |X_1\cap X_2|,$ and let $u$ be the minimum in $X_1\cup X_2$ such that $\ml_{X}(u)=2.$ Note that $N=|Y_1\cap Y_2|$ and $u$ is minimum such that $\ml_{Y}(u)=2.$ If $u=1,$ then define 
\begin{align*}
\Big{(} (X_1',X_2'),(Y_1',Y_2')\Big{)}:=\Ro_{(1,1)}\Big{(}(X_1,X_2),(Y_1,Y_2)\Big{)}, \mbox{ and }  \Ro_{N_{1}}:=\Ro_{(1,1)}. 
\end{align*}
Otherwise, define:
\begin{align*}
\Big{(}(X_1',X_2'),(Y_1',Y_2')\Big{)}&:=\Ro_{N_1} \Big{(}(X_1,X_2),(Y_1,Y_2)\Big{)},\\
\mbox{where, }~~\Ro_{N_1}&:=\Ro_{(2,1)}\circ \cdots \circ \Ro_{(u,u-1)}.
\end{align*}
Now $m_{X'}(1)=m_{Y'}(1)=2,$ and 
\begin{align*}
\det A\big{(}X_1'\big{|}Q_1\big{)}\det A\big{(}X_2'\big{|}Q_2\big{)} \leq \det A\big{(}Y_1'\big{|}J_1\big{)}\det A\big{(}Y_2'\big{|}J_2\big{)},
\end{align*}
for all $n\times n$ TN matrices $A.$
\item[$(3)$] If $\ml_{X'}(2)=\ml_{Y'}(2)=0,$ define 
\begin{align*}
\Big{(}(X_1'',X_2''),(Y_1'',Y_2'')\Big{)}:=\Ro_{(2,2)}\Big{(}(X_1',X_2'),(Y_1',Y_2')\Big{)}, \mbox{ and }\Ro_{N_{2}}:=\Ro_{(2,2)}. 
\end{align*}
Otherwise, let $a\in [n]$ be the least such that $\ml_{X'}(a)=0.$ This $a$ exists because $|X_1'\cap X_2'|=|X_1\cap X_2|\geq 1.$ Define:
\begin{align*}
\Big{(}(X_1'',X_2''),(Y_1'',Y_2'')\Big{)}&:=\Ro_{N_2} \Big{(}(X_1',X_2'),(Y_1',Y_2')\Big{)},\\
\mbox{where, }~~\Ro_{N_2}&:=\Ro_{(2,3)}\circ \cdots \circ \Ro_{(a-2,a-1)} \circ \Ro_{(a-1,a)}.
\end{align*}
Now $\ml_{X''}(2)=\ml_{Y''}(2)=0,$ $\ml_{X''}(1)=\ml_{Y''}(1)=2,$ and
\begin{align*}
\det A\big{(}X_1''\big{|}Q_1\big{)}\det A\big{(}X_2''\big{|}Q_2\big{)} \leq \det A\big{(}Y_1''\big{|}J_1\big{)}\det A\big{(}Y_2''\big{|}J_2\big{)},
\end{align*}
for all $n\times n$ TN matrices $A.$ This implies that for all even contiguous sets $S''\subseteq M''$
\begin{align}\label{prime}
\max\{|S''\cap R_1''|,|S'' \cap R_2''|\} \geq \max\{|S''\cap K_1''|,|S''\cap K_2''|\}, \mbox{ where,}
\end{align}
\begin{align*}
R_1''&:=X_1''\cup (2n+1-Q_2),~R_2'':=X_2'' \cup (2n+1-Q_1), \\
K_1''&:=Y_1''\cup (2n+1-J_2),~K_2'':=Y_2'' \cup (2n+1-J_1), \mbox{ and}\\
M''&:= (R_1''\setminus R_2'')\cup (R_2''\setminus R_1'').
\end{align*}

\item[$(4)$] If $M''=\emptyset,$ update:
\begin{align*}
\overline{X}_1=X_1'',~~ \overline{X}_2=X_2''(1,2), \mbox{ and } \overline{Y}_1=Y_1'',~~ \overline{Y}_2=Y_2''(1,2).
\end{align*}
Otherwise, if $M''\neq \emptyset,$ let $u\in M''$ be the minimum such that either $\ml_{X''}(u)=1$ or $\ml_{Q}(2n+1-u)=1.$ If $u\in (X_{i_1}'' \setminus X_{i_2}'' ) \cap (Y_{j_1}''\setminus Y_{j_2}'')$ or $u\in (2n+1-(Q_{i_1} \setminus Q_{i_2})) \cap (2n+1-(J_{j_1}\setminus J_{j_2}))$, then update: 
\begin{align*}
\overline{X}_{i_1}=X_{i_1}'',~~\overline{X}_{i_2}=X_{i_2}''(1,2),\mbox{ and } \overline{Y}_{j_1}=Y_{j_1}'',~~\overline{Y}_{j_2}=Y_{j_2}''(1,2).
\end{align*}
Define
\begin{align*}
\overline{R_1}&:=\overline{X_1}\cup (2n+1-Q_2),~\overline{R_2}:=\overline{X_2} \cup (2n+1-Q_1), \\
\overline{K_1}&:=\overline{Y_1}\cup (2n+1-J_2),~\overline{K_2}:=\overline{Y_2} \cup (2n+1-J_1), \mbox{ and}\\
\overline{M}&:= (\overline{R_1}\setminus \overline{R_2})\cup (\overline{R_2}\setminus \overline{R_1}).
\end{align*}
%
%
\item[$(5)$] Suppose $\overline{S}\subseteq \overline{M}$ is even contiguous which satisfies
$$
\max\{|\overline{S}\cap \overline{R_1}|,|\overline{S}\cap \overline{R_2}|\} < \max\{|\overline{S}\cap \overline{K_1}|,|\overline{S}\cap \overline{K_2}|\}.
$$
Then the following cases are possible:
\begin{itemize}
\item[•] If $\{1,2\}\subseteq \overline{S},$ then $S'':=\overline{S}\setminus\{1,2\}$ violates \eqref{prime}.
\item[•] Otherwise, if $\{2,3\}\subseteq \overline{S},$ then $S'':=\overline{S}\setminus \{2,3\}$ violates \eqref{prime}.
\item[•] Otherwise, $S'':=\overline{S}$ itself violates the necessary condition in \eqref{prime}.
\end{itemize}
Therefore,
\begin{align*}
\det A\big{(}\overline{X}_1\big{|}Q_1\big{)}\det A\big{(}\overline{X}_2\big{|}Q_2\big{)} \leq \det A\big{(}\overline{Y}_1\big{|}J_1\big{)}\det A\big{(}\overline{Y}_2\big{|}J_2\big{)},
\end{align*}
for all $n\times n$ TN matrices $A.$ Note that 
$
\overline{N}:=|\overline{X}_1 \cap \overline{X}_2| = |\overline{Y}_1 \cap \overline{Y}_2| = N-1.
$
This means following the steps from (2) onwards, we reduced the size of the common indices in the row index sets $X_1$ and $X_2$ and $Y_1$ and $Y_2.$ If needed review Definition~\ref{inv-op}, and define 
\begin{align*}
\Ro_{N}:=\Ro_{N_1}^{-1}\circ\Ro_{N_2}^{-1}\circ\Ro_{(2,1)}.
\end{align*}
\item[$(6)$] If $\overline{N}=0,$ i.e., $\overline{X_1}^{\mathsf{c}}=\overline{X_2}$ and $\overline{Y_1}^{\mathsf{c}}=\overline{Y_2},$ then set
\begin{align*}
P=\overline{X}_1&,~I=\overline{Y}_1,\mbox{ and }
\Ro:= \Ro_1 \circ \cdots \circ \Ro_{|P_1\cap P_2|},
\end{align*}
and go to step (7).
Otherwise, if $\overline{N}\geq 1,$ update
\begin{align*}
X_1=\overline{X}_1,~~X_2=\overline{X_2},~~
Y_1=\overline{Y}_1,~~Y_2=\overline{Y_2}, \mbox{ and go to step (2).}
\end{align*}
\item[$(7)$] Note that $\Ro \Big{(}\big{(}P,P^{\mathsf{c}}\big{)},\big{(}I,I^{\mathsf{c}}\big{)}\Big{)} = \Big{(}\big{(}P_1,P_2\big{)},\big{(}I_1,I_2\big{)}\Big{)},$ and there exist integers $k_r\geq 1,$ and integer pairs $(u_i,v_i)$ for $i \in [1,k_r]$ where each $u_i\in\{v_i-1,v_i,v_i+1\},$ such that
\begin{align*}
\Ro=\Ro_{(u_{k_r},v_{k_r})} \circ \cdots \circ \Ro_{(u_1,v_1)}.
\end{align*}
\item[$(8)$] We now have 
\begin{align*}
\det A\big{(}P\big{|}Q_1\big{)}\det A\big{(}P^{\mathsf{c}}\big{|}Q_2\big{)} \leq \det A\big{(}I\big{|}J_1\big{)}\det A\big{(}I^{\mathsf{c}}\big{|}J_2\big{)}, \mbox{ for all } n\times n \mbox{ TN }A. 
\end{align*}
\end{itemize}

Now proceed with step (1) with $(Q_1,Q_2),$ $(J_1,J_2)$ respectively taking the place of $(P_1,P_2),$ $(I_1,I_2),$ and $\big{(}P,P^{\mathsf{c}}\big{)},$ $\big{(}I,I^{\mathsf{c}}\big{)}$ taking the place of $(Q_1,Q_2),$ $(J_1,J_2).$ In the next iterations of the steps above, replace all ``$\Ro_{(*,*)}$'' with ``$\Co_{(*,*)}.$'' This provides us with a $\Co$ in step (6) such that 
$
\Co \Big{(}\big{(}Q,Q^{\mathsf{c}}\big{)},\big{(}J,J^{\mathsf{c}}\big{)}\Big{)} = \Big{(}\big{(}Q_1,Q_2\big{)},\big{(}J_1,J_2\big{)}\Big{)},
$ and for this $\Co,$ there exist integers $k_c \geq 1,$ and integer pairs $(u_j',v_j')$ for $j\in [1,k_c],$ where each $u_j'\in\{v_j'-1,v_j',v_j'+1\}$ such that,
\begin{align*}
\Co=\Co_{(u_{k_{c}}',v_{k_{c}}')} \circ \cdots \circ \Co_{(u_1',v_1')}.
\end{align*}
Hence,
\begin{align*}
\det A\big{(}P\big{|}Q\big{)}\det A\big{(}P^{\mathsf{c}}\big{|}Q^{\mathsf{c}}\big{)} \leq \det A\big{(}I\big{|}J\big{)}\det A\big{(}I^{\mathsf{c}}\big{|}J^{\mathsf{c}}\big{)}, \mbox{ for all } n\times n \mbox{ TN }A.
\end{align*}
Therefore, we have \eqref{Multi-thm-3-ineq-2}, and have found a composition of index-row/column operations
\begin{align*}
\Ro_{(u_{k_r},v_{k_r})} \circ \cdots \circ \Ro_{(u_1,v_1)}\circ
\Co_{(u_{k_{c}}',v_{k_{c}}')} \circ \cdots \circ \Co_{(u_1',v_1')},
\end{align*}
which when acts on inequality \eqref{Multi-thm-3-ineq-2} produces inequality \eqref{smallest-multi}. The converse is obvious using Lemma~\ref{Main-lemma-1}. This completes the proof.
\end{proof}

Having seen the proof of Theorem~\ref{Multi-thm-3}, one could reason that the so called hierarchy discussed in Remark~\ref{FGJ-Skan-comp-remark} is possible to obtain for the smallest multiplicative inequalities. We note that this is made possible to prove only because of the description of these inequalities in Theorem~\ref{FGJ-Skan-thm}, and any general reasoning behind this success is yet to be found.

\section{Min \& max multiplicities and the index-row/column operations}\label{minmax}

In this section, we discuss a generalization of Lemma~\ref{FGJ-Multi-lemma} and Lemma~\ref{Multi-lemma-1} (which are stated/proved in the context of multiplicative inequalities \eqref{Multi-ineq-1}) to certain refined multiplicity conditions that apply to the broader class of determinantal inequalities over TN matrices. 

For this purpose, consider a generic determinantal inequality over TN matrices: suppose $n\geq 1$ is an integer, and let $\mathcal{I}_1,\mathcal{I}_2$ be ordered nonnempty finite indexing sets, and $\gamma_k:\mathcal{I}_k \to \Z_{>0},$ for $k=1,2.$ Let $P_{ij},$ $Q_{ij},$ $I_{ij},$ $J_{ij}\subseteq [n],$ such that $|P_{ij}|=|Q_{ij}|$ and $|I_{ij}|=|J_{ij}|,$ for all $i\in \mathcal{I}_1\cup \mathcal{I}_2$ and $j\in [\gamma_1(i)]\cup[\gamma_2(i)].$ Suppose for $\lambda_{i}> 0$ and $\mu_{i}> 0,$
\begin{align}\label{minmax-lemma-1-ineq}
\sum_{i\in \mathcal{I}_1} \lambda_{i} \prod_{j=1}^{\gamma_1(i)} \det A\big{(}P_{ij}\big{|}Q_{ij}\big{)}\leq \sum_{i\in \mathcal{I}_2} \mu_{i} \prod_{j=1}^{\gamma_2(i)} \det A\big{(}I_{ij}\big{|}J_{ij}\big{)},
\end{align}
for all $n\times n$ totally nonnegative matrices $A.$ 

We introduce the required concepts of a minimum (min) and maximum (max) multiplicity and state the refinement of Lemma~\ref{FGJ-Multi-lemma} first:


\begin{defn}\label{minmax-defn-1}
Suppose $n\geq 1$ is an integer, and let $\mathcal{I}$ be an ordered nonnempty finite indexing set, and $\gamma:\mathcal{I} \to \Z_{>0}.$ Suppose $P_{ij}\subseteq [n]$ for all $i\in \mathcal{I}$ and $j\in [\gamma(i)].$ Define min and max multiplicities:
\begin{align*}
\ml^{\max}_{P}(u):= \max_{i \in \mathcal{I}} |\{k\in [\gamma(i)]:u \in P_{ik} \}|, \mbox{ and }~ \ml^{\min}_{P}(u):= \min_{i \in \mathcal{I}} |\{k\in [\gamma(i)]:u \in P_{ik} \}|,
\end{align*}
for all $u\in [n].$
\end{defn}

It is not difficult to check that $\ml_{X}^{\min}=\ml_{X}^{\max}=\ml_{X}$ for $X=P,Q,I,J$ in Definitions~\ref{Multi-ineq-defn-1} and \ref{minmax-defn-1}. Now we state the aforementioned generalization.

\begin{ulemma}\label{minmax-lemma-1}
Suppose $n\geq 1$ is an integer, and inequality~\eqref{minmax-lemma-1-ineq} holds. Then
\begin{align*}
\ml^{\max}_{P}(u) \leq \ml^{\max}_{I}(u) &\mbox{ and }~\ml^{\max}_{Q}(u) \leq \ml^{\max}_{J}(u), \mbox{ and}\\
\ml^{\min}_{P}(u) \geq \ml^{\min}_{I}(u) &\mbox{ and }~\ml^{\min}_{Q}(u) \geq \ml^{\min}_{J}(u),  \mbox{ for all } u\in [n].
\end{align*}
\end{ulemma}

Lemma~\ref{minmax-lemma-1} clearly simplifies to Lemma~\ref{FGJ-Multi-lemma} since the min and max multiplicities become equal; below we provide the verification.

\begin{proof}[Proof of Lemma~\ref{FGJ-Multi-lemma}]
Note that for multiplicative inequalities \eqref{Multi-ineq-1} the cardinalities $|\mathcal{I}_1|=|\mathcal{I}_2|=1,$ where $\mathcal{I}_1,\mathcal{I}_2$ are as in Lemma~\ref{minmax-lemma-1}. This implies $\ml_{X}^{\min}=\ml_{X}^{\max}=\ml_{X}$ for $X=P,Q,I,J.$ Hence, using Lemma~\ref{minmax-lemma-1}, we obtain the conclusions in Lemma~\ref{FGJ-Multi-lemma}.
\end{proof}

\begin{proof}[Proof of Lemma~\ref{minmax-lemma-1}]
Define the following for all $n\times n$ TN matrices $A:$
\begin{align*}
F_{\mathcal{I}_1}(A):=\sum_{i\in \mathcal{I}_1} \lambda_{i} \prod_{j=1}^{\gamma_1(i)} \det A\big{(}P_{ij}\big{|}Q_{ij}\big{)}, \mbox{ and }~ F_{\mathcal{I}_2}(A) := \sum_{i\in \mathcal{I}_2} \mu_{i} \prod_{j=1}^{\gamma_2(i)} \det A\big{(}I_{ij}\big{|}J_{ij}\big{)}.
\end{align*}
Now consider the  $n\times n$ TN matrix $A$ such that each $w_{j,k}=w_{j,k+1}'=1$ and $D=\I$ in \eqref{TN-classification-eqn}. Let $u\in [n],$ and suppose $B_u(w)=D_{u}(w)A,$ where $D_{u}(w):=\diag(1,\ldots, w ,\ldots, 1)$ with $u^{\mbox{th}}$ diagonal entry being $w>0$ and the remaining diagonal entries being $1.$

Observe that $F_{\mathcal{I}_1}(B_u(w))$ and $F_{\mathcal{I}_2}(B_u(w))$ are polynomials in $w$ of degrees $\ml_{P}^{\max}(u)$ and $\ml_{I}^{\max}(u)$, respectively. Since $F_{\mathcal{I}_1}(B_u(w))\leq F_{\mathcal{I}_2}(B_u(w))$ for all $w>0,$ if $\ml_{P}^{\max}(u)>\ml_{I}^{\max}(u),$ then for large values of $w>0$ we contradict \eqref{minmax-lemma-1-ineq}. Therefore $\ml_{P}^{\max}(u) \leq \ml_{I}^{\max}(u)$ for all $u\in [n].$ Similarly we can prove that $\ml_{Q}^{\max}(u) \leq \ml_{J}^{\max}(u)$, for all $u\in [n]$ be setting $B_u(w)=AD_{u}(w)$.

On the other hand, $w^{\ml_{I}^{\max}(u)}F_{\mathcal{I}_1}(B_u(1/w))$ and $w^{\ml_{I}^{\max}(u)}F_{\mathcal{I}_2}(B_u(1/w))$ are polynomials in $w$ of degrees $\ml_{I}^{\max}(u)-\ml_{P}^{\min}(u)$ and $\ml_{I}^{\max}(u)-\ml_{I}^{\min}(u)$, respectively. If $\ml_{P}^{\min}(u) < \ml_{I}^{\min}(u),$ since $w^{\ml_{I}^{\max}(u)}F_{\mathcal{I}_1}(B_u(1/w)) \leq w^{\ml_{I}^{\max}(u)}F_{\mathcal{I}_2}(B_u(1/w))$ for all $w>0,$ for large values of $w>0$ we contradict \eqref{minmax-lemma-1-ineq}. Therefore $\ml^{\min}_{P}(u) \geq \ml^{\min}_{I}(u)$ for all $u\in [n].$ Similarly, we can prove the remaining inequality by considering $B_u(w)=AD_{u}(w)$.
\end{proof}

To summarize: we introduce the concepts of min and max multiplicities corresponding to any row/column index in any determinantal inequality over TN matrices. We see that the min and max multiplicities coincide for the (simpler) case of multiplicative inequalities \eqref{Multi-ineq-1}. This generalizes the concept of multiplicity in Lemma~\ref{FGJ-Multi-lemma} to the min and max multiplicities for the extended class of inequality over TN matrices in Lemma~\ref{minmax-lemma-1}. For the next refinement (of Lemma~\ref{Multi-lemma-1}), if needed, recall Definition~\ref{row-op} before going further:

\begin{defn}\label{minmax-defn-2}
Suppose $n\geq 1$ is an integer, and let $\mathcal{I}$ be an ordered nonnempty finite indexing set, and $\gamma:\mathcal{I} \to \Z_{>0}.$ Suppose $P_{ij}\subseteq [n]$ for all $i\in \mathcal{I}$ and $j\in [\gamma(i)].$ Define:
\begin{align*}
\ml_{P}(u,v):= \max_{i \in \mathcal{I}} |i(u,v)|, \mbox{ for all consecutive }u,v\in [n].
\end{align*}
\end{defn}

\begin{ulemma}\label{minmax-lemma-2}
Suppose $n\geq 1$ is an integer, and inequality~\eqref{minmax-lemma-1-ineq} holds. Then
\begin{align*}
\ml_{P}(u,v)\leq \ml_I(u,v) \mbox{ and }~\ml_{Q}(u,v)\leq \ml_J(u,v),\mbox{ for all consecutive }u,v\in [n].
\end{align*}
\end{ulemma}

The proof of Lemma~\ref{minmax-lemma-2} proceeds verbatim as in the proof of Lemma~\ref{Multi-lemma-1} and we leave the details to the reader. Lemma~\ref{minmax-lemma-2} provides the next level of multiplicity conditions -- which arise from the index-row/column operations -- necessary for validity of all inequalities of the form (\ref{minmax-lemma-1-ineq}) over TN matrices. Now, recall that Algorithm~\ref{algo} was devised based on Lemma~\ref{Multi-lemma-1} (for multiplicative inequalities), therefore a similar algorithm can be constructed following Lemma~\ref{minmax-lemma-2} (applicable to all determinantal inequalities). Finally, comments similar to Remark~\ref{algo-succeed-row-op-remark} can be made about this refined algorithm. We leave the required details to the reader.

\section{Inferences, questions, and the Barrett--Johnson inequalities}\label{Section-conclusion}

In this paper we have set out to constructively utilize two important properties for the classes of TN matrices, namely multiplicative closure and the existence of a bidiagonal factorization, to aid in potentially exhibiting new relationships among minors for this class of matrices. Specifically, we systematically tracked the effect on a determinantal relationship upon multiplication by an elementary bidiagonal TN matrix. Bearing in mind some careful analysis and accounting, we derived, based on both classical and more modern determinantal relationships, new such relationships leading to some novel expressions involving certain minors of TN matrices. In addition, we have expressed this impact on certain relationships as index-row/column operations on the minors directly. This view point, in some sense, exposes a more general interpretation on the possible derived determinantal relationships.

Our main contributions along these lines (in the order of presentation) are Theorems~\ref{Add-thm-1}, \ref{Add-thm-2}, and \ref{Add-thm-3}, which present refinements of some classical relationships due to Laplace, Karlin, and Gantmacher--Krein, in addition to unravelling an undiscovered class of additive inequalities over TN matrices. Moreover, these results account for all known additive inequalities (except for Theorem~\ref{Skan-Sos}) associated with totally nonnegative matrices. Theorems \ref{Multi-thm-2} and \ref{Multi-thm-3} represent our main results within the context of multiplicative inequalities for the class of TN matrices. All of these inequalities appear to be derivable from higher/complementary inequalities (except for Theorem~\ref{Add-thm-2}, where the complementary ones do not exist). Furthermore, they all exhibit the ``converse'' of the universally applicable Lemma~\ref{Main-lemma-1} -- the additive ones show it trivially, and the smallest multiplicative with some work. 

Throughout this work, we also develop needed tools and machinery that built a framework for formally expressing the new relationships exposed by our analysis (see Lemmas \ref{Main-lemma-1}, \ref{Main-lemma-2}, \ref{Main-lemma-3}, \ref{Add-thm-3-lemma}, and \ref{Multi-lemma-1}, along with Lemmas~\ref{minmax-lemma-1},~\ref{minmax-lemma-2}, and Theorem \ref{Multi-thm-1}). We also provide relevant comments as additional study material (see Remarks~\ref{Hierarchy-rem-0},~\ref{MainLemma-converse},~\ref{Main-lemma-3-remark},~\ref{Hierarchy-rem-1},~\ref{Add-thm-3-remark},~\ref{Column-op-rem},~\ref{Hierarchy-remark},~\ref{algo-succeed-row-op-remark},~\ref{algo-limit-remark},~\ref{Multi-thm-2-significance},~\ref{FGJ-Skan-comp-remark}, and~\ref{Multi-thm-3-significance}, along with Examples~\ref{multi-example-1}, and \ref{multi-example-2}). We now conclude this paper with a few follow up questions.

\begin{uquestion}\label{Q1}
Complete the refinement of the generalized Laplace identity (see \cite{M}) -- which we partially achieved in Theorem~\ref{Add-thm-3} -- along the lines of Theorem~\ref{Add-thm-1} (or Gantmacher--Krein's \eqref{GK-ineq}).
\end{uquestion}

\begin{uquestion}\label{Q2}
We saw that proving the so-called converse of Lemma~\ref{Main-lemma-1} would seem too ambitious in general, and that all the smallest multiplicative inequalities comply with this converse (see Theorem~\ref{Multi-thm-2} and Remark~\ref{Multi-thm-2-significance}). The universality of Lemma~\ref{Main-lemma-1} naturally motivates to seek its converse, or (at least) more inequalities for which this converse holds, hence the problem.
\end{uquestion}

For Question~\ref{Q3}, we illustrate the theory developed here applied to a recently exhibited and interesting Barrett--Johnson determinantal inequality for TN matrices.
\begin{theorem}[Skandera--Soskin \cite{Skan-Sos}]\label{Skan-Sos}
Let $n,r\geq 1$ be integers and $A$ be an $n\times n$ totally nonnegative matrix. Suppose $(\lambda_1,\ldots,\lambda_r)$ and $(\mu_1,\ldots,\mu_r)$ are nonincreasing sequences of nonnegative integers that sum up to $n.$ If $\sum_{k=1}^{l}\lambda_k \leq \sum_{k=1}^{l}\mu_k$ for all $l\in [1,r],$ then
\begin{align}\label{BJ-old}
\lambda_1 ! \cdots \lambda_r ! \sum_{I} \prod_{k=1}^{r} \det A\big{(}I_k\big{)}\geq\mu_1 ! \cdots \mu_r ! \sum_{J}\prod_{k=1}^{r} \det A \big{(}J_k\big{)}
\end{align}
where the sums are taken over collection of partitions 
$$
I:=(I_1,\ldots,I_r)\mbox{ and } J:=(J_1,\ldots,J_r)\mbox{ of }[n],
$$ 
such that for each $k\in [1,r],$ $|I_k|=\lambda_k,$ and $|J_k|=\mu_k.$ 
\end{theorem}

When we subsequently apply $\Ro_{(u,v)}$ and $\Co_{(u,v)},$ for consecutive $u,v\in [n],$ on \eqref{BJ-old}, we obtain a novel inequality:

\begin{ucor}\label{BJ-thm-new}
Under the hypotheses in Theorem~\ref{Skan-Sos}, for consecutive $u,v\in [n],$
\begin{align}\label{BJ-ineq-new}
\lambda_1 ! \cdots \lambda_r ! \sum_{I{(u,v)}} \prod_{k=1}^{r} \det A\big{(}I_k(u,v)\big{)} \geq \mu_1 ! \cdots \mu_r ! \sum_{J{(u,v)}}\prod_{k=1}^{r} \det A \big{(}J_k(u,v)\big{)}
\end{align}
where the sums are taken over collection of partitions 
$$
I{(u,v)}:=(I_1,\ldots,I_r)\mbox{ and } J{(u,v)}:=(J_1,\ldots,J_r)\mbox{ of }[n],
$$ 
such that for each $k\in [1,r],$ $|I_k|=\lambda_k,$ $|J_k|=\mu_k,$ $\{u,v\}\not\subseteq I_{k},$ and $\{u,v\}\not\subseteq J_{k}.$
\end{ucor}

\begin{uquestion}\label{Q3}
The Barrett--Johnson inequality~\eqref{BJ-old} certainly sits at the top of the hierarchy among inequalities (see Remarks~\ref{Hierarchy-rem-0},~\ref{Hierarchy-remark}, and~\ref{FGJ-Skan-comp-remark}), and the discussion above clearly hints at the existence of several nontrivial inequalities for TN matrices that result by simply applying different compositions of the index-row/column operations $\Ro_{(*,*)}$/$\Co_{(*,*)};$ we listed just one of those in Corollary~\ref{BJ-thm-new}. However, it does not look very straightforward to identify all of these at once, hence the problem.
\end{uquestion}

\section*{Acknowledgements} 
We would like to thank Apoorva Khare for carefully going through this work and for providing valuable feedback that helped to improve the manuscript. We also thank the American Institute of Mathematics (CalTech campus) for their hospitality and stimulating environment during a workshop on {Theory and Applications of Total Positivity,} held in July 2023. This allowed us to interact fruitfully with several other researchers. In particular, we thank Alan Sokal for the suggestion to change the nomenclature from ``row operations'' to ``index-row operations'', Melissa Sherman-Bennett for stimulating discussions on the (totally nonnegative) Grassmannian, and Lauren Williams for the suggestion to add Remark~\ref{alg-rel}. These suggestions provided further clarity in the exposition of the manuscript.

S.M. Fallat was supported in part by an NSERC Discovery Research Grant, Application No.: RGPIN-2019-03934. P.K. Vishwakarma was supported in part by a Pacific Institute for the Mathematical Sciences (PIMS) Postdoctoral Fellowship, and he also acknowledges the Canadian Queen Elizabeth II Diamond Jubilee Scholarship that supported his first research visit to Regina $($in $2019)$ where initial discussions took place.

%
%
%


\addtocontents{toc}{\protect\setcounter{tocdepth}{1}}

\vspace*{2cm}

\end{document}